\documentclass[10pt]{article}
\usepackage{amssymb,amsmath,textcomp}
\usepackage[T1]{fontenc}
\usepackage[utf8]{inputenc}
\usepackage{enumerate}
\usepackage{hyperref}
\usepackage{mathrsfs}
\usepackage{amsthm}
\usepackage{xcolor}
\usepackage[multiple]{footmisc}

\title{Lipschitz stability for Bayesian inference in porous medium tissue growth models
}

\author{Tomasz D\k{e}biec$^1$, Piotr Gwiazda$^{2,3}$,  Błażej Miasojedow$^1$, Katarzyna Ryszewska$^3$, \\ Zuzanna Szyma\'{n}ska$^2$ and Aneta Wr\'{o}blewska-Kami\'{n}ska$^3$}

\newtheorem{theo}{Theorem}

\newtheorem{prop}{Proposition}
\newtheorem{lemma}{Lemma}
\newtheorem{coro}{Corollary}
\newtheorem{remark}{Remark}

\def\divv{\operatorname {div}}
\def\sign{\operatorname{sign}}
\def\supp{\operatorname{supp}}

\newcommand{\eqq}[2]{\begin{equation}  #1  \label{#2}\end{equation}    }
\newcommand{\hd}{\hspace{0.2cm}}

\newcommand{\m}[1]{\mbox{#1}}
\newcommand*{\norm}[1]{\left\Vert{#1}\right\Vert}
\newcommand*{\abs}[1]{\left\vert{#1}\right\vert}
\newcommand{\vf}{\varphi}
\newcommand{\ve}{\varepsilon}
\newcommand{\R}{\mathbb{R}}
\newcommand{\nic}[1]{ }

\begin{document}

\maketitle
\abstract{We consider a macroscopic model for the dynamics of living tissues incorporating
pressure-driven dispersal and pressure-modulated proliferation. Given a power-law constitutive relation between the pressure and cell density, the model can be written as a porous medium equation with a growth term. We prove Lipschitz continuity of the mild solutions of the model with respect to the diffusion parameter (the exponent $\gamma$ in the pressure-density law) in the $L_1$ norm. While of independent analytical interest, our motivation for this result is to provide a vital step towards using Bayesian inverse problem methodology for parameter estimation based on experimental data -- such stability estimates are indispensable for applying sampling algorithms which rely on the gradient of the likelihood function.
}

\vskip 1cm

\footnotetext[1]{University of Warsaw, Institute of Applied Mathematics and Mechanics, Banacha 2, 02-097 Warsaw, Poland (t.debiec@mimuw.edu.pl, b.miasojedow@mimuw.edu.pl).}

\footnotetext[2]{University of Warsaw, Interdisciplinary Centre for Mathematical and Computational Modelling, Tyniecka 15/17, 02-630 Warsaw, Poland (pgwiazda@mimuw.edu.pl, zk.szymanska@uw.edu.pl)}

\footnotetext[3]{Institute of Mathematics, Polish Academy of Sciences,
\'{S}niadeckich 8, 00-656 Warsaw, Poland (katarzyna.ryszewska@pw.edu.pl, awrob@impan.pl)}

\section{Introduction and main results}
\label{sec:Intro}

Mathematical modelling plays a crucial role in understanding the complex dynamics of tumour growth. Among the various modeling approaches, macroscopic density-based partial differential equations provide a powerful framework for describing the evolution of tumour cell populations over time and space~\cite{ByrneDrasdo, Friedman2007, Goriely}. These models often incorporate mechanical effects by relating tumour cell density to pressure through constitutive relations inspired by fluid and porous media mechanics. One widely used class of models assumes a nonlinear relation between the pressure and density, where the pressure follows a power-law dependence on the density with an exponent $\gamma$, leading to porous-medium type equations with a pressure-dependent proliferation term~\cite{PQV14}.  The macroscopic density of cells, $u=u(t,x)$, is assumed to satisfy the following equation
\begin{equation*}
	\partial_t u + \divv(uv) = uG(p),
\end{equation*}
where the velocity obeys Darcy's law $v = -\nabla p$ and the pressure satisfies
$$p = \frac{\gamma}{\gamma-1}u^{\gamma-1}$$
for some $\gamma > 1$. The reaction rate $G$ is a given function, whose properties will be specified below.
Substituting the velocity law into the equation, it can be rewritten as a porous medium type equation:
\begin{equation}\label{eq:PME_intro}
	\partial_t u = \Delta(u^\gamma)  + uG(p).
\end{equation}
The diffusion exponent $\gamma$ characterises the pattern of tumour expansion, reflecting both the speed and morphology of the evolving tumour boundary.

For these models to be effectively applied in clinical or biological settings, it is essential to calibrate them against observed data. Experimental measurements such as tumour morphology, expansion rates, and imaging data provide valuable information about tumour progression. However, the constitutive relation parameter,  which governs the tumour’s mechanical response, is often unknown and varies across different tumour types and conditions. Therefore, inferring the value of $\gamma$ from observational data is crucial to ensuring that the model accurately reflects real biological behavior. Traditional calibration approaches rely on direct fitting techniques or optimization-based inverse problems, but these methods can be sensitive to noise and may not adequately capture uncertainties in the inferred parameters.

The observed data set being finite and noisy, we face an underdetermined, and often ill-posed, inverse problem, which we would like to address using the framework of Bayesian statistics~\cite{Stuart}.  To this end, we shall investigate the stability of problem~\eqref{eq:PME_intro} with respect to the parameter $\gamma$.
The Bayesian approach provides a probabilistic characterization of uncertainty, allowing us to quantify the confidence in the inferred values of the model parameters and incorporate prior knowledge about their plausible range.  By leveraging observed tumour density distributions at different time points, one can construct a posterior distribution for $\gamma$, which reflects both the likelihood of the data given a specific model realisation and the prior assumptions on the parameter.
More precisely, Bayesian inference is based on posterior distribution together with the conditional distribution of parameters given the observed data, according to the formula
\begin{equation}\label{eq:Bayes}
	\pi(\gamma|D) = \frac{\pi(D|\gamma)\pi(\gamma)}{\int_{\Gamma}\pi(D|\gamma)\;\mathrm{d}\pi(\gamma)},
\end{equation}
where $D$ denotes the collected data, $\gamma$ is a given parameter, or more generally a vector of parameters, taking values in the parameter space $\Gamma$. The posterior $\pi(\gamma|D)$ is the probability density of state $\gamma$ given data $D$,  $\pi(\gamma)$ is the prior probability density, i.e., the probability density of the values $\gamma$ without any knowledge of the data, and $\pi(D|\gamma)$ is the likelihood function, which quantifies the probability of observing data $D$ when the parameters take values $\gamma$. Unsurprisingly, it is practically impossible to obtain a closed analytical formula for the posterior distribution, and hence sampling methods have to be employed. We refer the reader to the recent works~\cite{Falco, M2AN, gwiazda2023, Kahle, Lipkova, zs2021} for examples of the use of the Bayesian inverse problem framework for models arising in tumour modelling. In particular, in~\cite{M2AN} the authors analyse the same model~\eqref{eq:PME_intro} and study the data assimilation problem to estimate the initial condition and the proliferation rate -- however, assuming that the diffusion parameter $\gamma$ is known.

A widely used method for sampling from complex posterior distributions in Bayesian inference is the Metropolis-Hastings Markov Chain Monte Carlo (MCMC) algorithm~\cite{GelfandSmith,Hastings,Metropolis}. It is based on generating a Markov chain whose stationary distribution corresponds to the target posterior, allowing for approximations of expected values and credible intervals. The method involves proposing candidate samples based on a chosen transition kernel and accepting or rejecting these samples based on an acceptance ratio derived from the posterior probability. While powerful, Metropolis-Hastings MCMC can suffer from slow convergence and high autocorrelation, particularly when sampling high-dimensional or strongly correlated distributions.  As such,  the convergence of the algorithm may actually fail in practice, given a reasonable computation time.
Additionally, the efficiency of the algorithm is highly sensitive to the choice of proposal distribution, making parameter tuning a critical challenge.

To address some of the limitations of Metropolis-Hastings MCMC, the Langevin-based approach incorporates gradient information into the proposal step by simulating the Langevin equation from stochastic calculus:
\begin{equation*}
	\mathrm{d} \gamma(t) = -\nabla V(\gamma(t))\mathrm{d}t + \sqrt{2}\, \mathrm{d}B(t),
\end{equation*}
where $B$ is the standard Brownian motion and $V = -\log\pi(D|\gamma) - \log\pi(\gamma)$ is the negative log-posterior (acting like an energy function). The Metropolis-Adjusted Langevin Algorithm (MALA) modifies the standard random walk proposal by incorporating the gradient of the posterior density, thereby guiding the Markov chain toward high-probability regions more efficiently. This results in improved convergence rates and reduced autocorrelation, making the approach particularly useful for Bayesian inference problems where gradients can be computed efficiently~\cite{RobertsTweedie}. In the context of tumour growth modeling, MALA can significantly enhance the efficiency of Bayesian inversion for $\gamma$,  allowing for more accurate and computationally feasible parameter inference.

To understand the overall evolution of the probability density, whose individual sample paths are described by the Langevin equation,  one studies the Fokker-Planck equation:
\begin{equation*}
	\partial_t \pi(\gamma, t) = \divv(\pi(\gamma,t)\nabla V(\gamma)) + \Delta\pi(\gamma,t).
\end{equation*}
Crucially, the steady state of this equation corresponds (up to normalisation constant) to the posterior distribution:
\begin{equation*}
	e^{-V(\gamma)} = \pi(D|\gamma)\pi(\gamma).
\end{equation*}
Since the solutions of the Fokker-Planck equation do converge to the equilibrium state, this implies that if we run the Langevin dynamics for a sufficiently long time, the probability density of samples will converge to the true posterior distribution.  Let us mention that other deterministic PDE models can be considered, each having their own advantages -- for instance, the blob method for a nonlocal approximation of the Fokker-Planck equation~\cite{CKP}, or the Stein variational gradient descent method~\cite{CS,KorbaEtAl,LiuWang}.

A crucial aspect of implementing Bayesian inference for the exponent $\gamma$  is obtaining a quantitative estimate of the stability of the underlying PDE model~\eqref{eq:PME_intro} with respect to the diffusion exponent.
If small variations in $\gamma$  lead to significant changes in the solution of the PDE,  the inference process may be highly sensitive to noise in the observational data. Conversely, if the model exhibits robustness to variations in $\gamma$, the inferred posterior distribution will be more stable and reliable.  By analyzing the stability properties of the model, we can ensure that the Bayesian inversion framework provides meaningful and well-posed estimates of $\gamma$, ultimately leading to more robust predictions of tumour growth dynamics.
More precisely,  a common feature of the above PDE methods for sampling algorithms is that they are defined through the gradient of the likelihood function.  For the likelihood functions which are not $C^1$, or at least Lipschitz,  we can apply some smoothed variant of the method,  but at the expense of generating an additional error, and ultimately resulting in a significantly slower rate of convergence. It is therefore our main goal in this paper to establish an explicit Lipschitz continuous dependence of the solutions of~\eqref{eq:PME_intro} on the exponent $\gamma$. Naturally, from the practical viewpoint, it is of equal interest to estimates other parameters in the model (e.g., in the proliferation rate). We focus here solely on the exponent $\gamma$ as it is the most analytically challenging.

Moreover, we have to bear in mind that in practice the posterior distribution $\pi(\gamma|D)$ is not calculated from exact solutions of the PDE evolving from exactly given data. Rather, it is approximated using numerical solutions, while the data is subject to measurement uncertainties.  The latter is taken into account in the model of the data, i.e., we assume that
the tumour mass is known at a finite set of disjoint spatial locations $\{Q_k\}_{k=1}^N$ at a discrete set of time points $\{t_j\}_{j=1}^M$.  In particular, we are given the operators
\begin{equation*}
	m_{kj}:u\mapsto m_{kj}(u),\quad m_{kj}(u) = \int_{Q_k} u(t_j,x)\;\mathrm{d} x.
\end{equation*}
The observed data is then modelled as
\begin{equation*}
	y = \mathcal{M}(\gamma) + \eta,
\end{equation*}
where $\mathcal{M}:\Gamma\to\R^{MN}_+$ is the forward operator obtained from the composition of the solution operator $\gamma\mapsto u_\gamma$ to~\eqref{eq:PME_intro} with the matrix-valued operator $m = (m_{kj})$ mapping $u_\gamma\mapsto m(u)$; and $\eta \sim\mathcal{N}(0,\Sigma)$ is an observational noise, distributed normally with a give covariance matrix $\Sigma$. Assuming a Gaussian prior distribution for the parameter $\gamma\sim\mathcal{N}(m_0,\sigma_0^2)$,  the posterior distribution is given by
\begin{equation}\label{eq:Likelihood}
	\pi(\gamma|D) \propto e^{-V(\gamma)},\quad V(\gamma) =\frac12|y-\mathcal{M}(\gamma)|_{\Sigma}^2 +\frac{|\gamma - m_0|^2}{2\sigma_0^2}.
\end{equation}
The former term represents the likelihood of observing data $y$ given the value $\gamma$ of the parameter. It penalises deviations of the model prediction from the actual observations, scaled by the noise covarriance $\gamma$,
where $|y-\mathcal{M}(\gamma)|_{\Sigma}^2 = (y-\mathcal{M}(\gamma))^T\Sigma^{-1}(y-\mathcal{M}(\gamma))$. The latter term represents the prior knowledge about $\gamma$ and penalises its deviations from the mean $m_0$ scaled by the prior variance $\sigma_0^2$.

\subsection{Main results and assumptions}
As outlined above, our main concern in this paper is to establish Lipschitz continuity of the solutions to~\eqref{eq:PME_intro} with respect to the diffusion parameter $\gamma$. More specifically, we will work with mild solutions in the $L^1$-topology.
Our main result for equation~\eqref{eq:PME_intro} reads, somewhat informally, as follows.
\begin{theo}\label{thm:MainIntro}
Let $\gamma_1, \gamma_2 \in \Gamma \subset (1,\infty)$ and let $u_1, u_2$ be the corresponding mild solutions to equation~\eqref{eq:PME_intro} in $\mathbb{R}^d \times (0,T)$. The following stability estimate holds:
\begin{align*}
    \norm{u_1(t)-u_2(t)}_{L_1} \leq e^{tG(0)}\norm{u^0_1-u_2^0}_{L_1} + C_1te^{tG(0)}\norm{G_2-G_1}_{L^\infty} + C_2(\sqrt{dt}+t)e^{tG(0)}|\gamma_2-\gamma_1|,
\end{align*}
for any $t\geq 0$, where the constants $C_1$, $C_2$ depend on $\gamma_1, \gamma_2$ and the initial data; and $G(0):=\max{(G_1(0),G_2(0))}$.
\end{theo}

In fact, we shall first establish a more general stability result for an abstract degenerate parabolic filtration equation with growth and then deduce the above theorem as a special case -- see Theorem~\ref{thm:Main1} and Corollary~\ref{thm:Main2} in Section~\ref{sec:Stability} for precise statements.
To this end, we discuss the following problem:
\eqq{
\partial_{t} u = \Delta \varphi (u) +(u)_{+
}G(p(u))\hd  \m{ in } \mathbb{R}^d \times (0,T), \hd  u(0) = u^{0}  \m{ in } \mathbb{R}^{d},
}{r1}
under the following assumptions
\eqq{
\vf \in C^{1}(\mathbb{R}), \hd \vf(0) = 0, \hd \vf \m{ is increasing},
}{A1}
\eqq{
G \in C([0,\infty)) \cap C^1((0,\infty)),  \hd G(0) > 0,  \hd G' \leq 0 \m{ on } (0,\infty),
}{A2}
\eqq{
p\in C(\mathbb{R}), \hd  \hd p  \m{ is nonnegative  and nondecreasing}.
}{A3}
Here, and throughout, $(u)_+ = \max(0,u)$ denotes the positive part of a function.
We are interested in the stability properties with respect to given functions $\vf,G$, and $p$ for mild solutions to (\ref{r1}).
Then, we apply the obtained result to the special case
\[
\vf_{i}(t) = \sign(t)|t|^{\gamma_i}, \hd p_{i}(t) = \frac{\gamma_i}{\gamma_i - 1}|t|^{\gamma_i - 1}, \hd \m{ for } i=1,2, \hd \gamma_2 > \gamma_1 >1,
\]
to obtain Lipschitz continuity of mild solutions to~\eqref{eq:PME_intro} in $L_{1}$- norm with respect to parameter $\gamma$, as stated in Theorem~\ref{thm:MainIntro}.

The mild solutions to equation~\eqref{r1} can be constructed, for instance, via the implicit time discretisation scheme:
\begin{equation*}
    u(t) = \lim_{n\to\infty}\left[\left(I + \frac{t}{n}A\right)^{-1}\right]^nu^0,
\end{equation*}
where $A(u):=-\Delta\varphi(u)-(u)_+G(p(u))$. We refer the reader to~\cite{Pavel} and~\cite[Chapter 10]{Vazquez} for detailed discussions. As mentioned previously, an important source of error in calculating the posterior distribution $\pi(\gamma|D)$ comes from it being computed from an approximate solution, for instance, the implicit time-discretised approximation, rather than the true analytical solution of the PDE. Therefore, it is desirable that the numerical scheme satisfies the same stability estimate as the analytical solutions of the equation.
To this end, we formulate a suitable stability estimate, which arises as a natural by-product of the proof of Theorem~\ref{thm:Main1}.
\begin{lemma}\label{thm:Numerical}
Let $\gamma_1, \gamma_2 >1$ and for $i=1,2$ let
$u_i^n$ be the piecewise constant function
\begin{equation*}
    u_i^n(t) = u_i^n \quad\text{for $t\in [n\tau, (n+1)\tau)$},
\end{equation*}
where
\begin{equation*}
    u^n_i = (I+\tau A)^{-1} u_i^{n-1},\quad\text{and}\quad u_i^0 = u_i^0.
\end{equation*}
Then, the same estimate as in Theorem~\ref{thm:MainIntro} holds for the difference $u_1^n(t) - u_2^n(t)$.
\end{lemma}

Since the functions $u^n$ converge as $n\to\infty$ to the unique mild solution of~\eqref{eq:PME_intro} strongly in $L_1$, we can compare the posterior $\pi^n(\gamma|D)$ arising from a numerical approximation and the one arising from the analytical solution. Indeed, using the forward operator $\mathcal{M}$ as in~\eqref{eq:Likelihood} and the form of the operator $m$, it easily follows that $\pi^n(\gamma|D)$ converges to $\pi(\gamma|D)$ in $TV$. (Granted, using the implicit time discretisation still leads to solving an infinite-dimensional problem and is not optimal in terms of computation time, but we leave employing a more practical numerical method for further study.) More precisely, we have the following result.
\begin{lemma}(Stability of the posterior distribution)
    Suppose that the parameter space $\Gamma$ is bounded. The following estimate holds
    \begin{equation}\label{Like}
        \norm{\pi(\gamma|D)-\pi^n(\gamma|D)}_{TV} \leq C\frac{1}{\sqrt{n}}.
    \end{equation}
\end{lemma}
\begin{proof}
    Using formula~\eqref{eq:Bayes} we obtain the following estimate
    \begin{align}\label{Like1}
        \norm{\pi(\gamma|D)-\pi^n(\gamma|D)}_{TV} \leq 2\frac{\int_{\Gamma}|\pi(D|\gamma)-\pi^n(D|\gamma)|\;\mathrm{d}\pi(\gamma)}{\int_{\Gamma}\pi(D|\gamma)\;\mathrm{d}\pi(\gamma)}.
    \end{align}
    Next, from~\eqref{eq:Likelihood} we deduce that the likelihood function potential is bounded both from below and from above. Finally, we observe that
    \begin{equation}\label{Like2}
        |V(\gamma) - V^n(\gamma)| \leq C\sum_{i,j}|m_{ij}(u_\gamma) - m_{ij}(u^n_\gamma)| \leq C \norm{u-u^n}_{L_1(\R^d)} \leq C\frac{1}{\sqrt{n}},
    \end{equation}
    where the last inequality follows from Theorem~\ref{CL}, cf.~\cite{CL}, which provides the explicit rate for the convergence $u^n\to u$, and Proposition~\ref{Atheo}, which guarantees that the operator $A(u)=-\Delta \vf(u) - (u)_{+}G(p(u))$ is $G(0)$-accretive and satisfies the assumptions of Theorem~\ref{CL}. Consequently, since $
    \pi(\gamma)$ is a probability measure, from~\eqref{Like1} and~\eqref{Like2} we deduce~\eqref{Like}.
\end{proof}

\subsection{Structure of the paper}
The rest of the paper is devoted to proving Theorem~\ref{thm:MainIntro} and is organised as follows. In Section \ref{sec:RegularResolvent} we regularise problem~\eqref{r1} by considering a smoother function $\varphi$. In this setting we analyse the resolvent operator $(I+\tau A)^{-1}$. In particular, we use the doubling of variables method to derive a localised stability estimate for differences of functions $(I+\tau A_i)^{-1}f_i$, see~\eqref{estepsilon}. This requires one of the $f_i$ to belong to the $BV$ space -- later, such regularity will be required of one of the initial data.
In Section~\ref{sec:Resolvent} we translate the key results of Section~\ref{sec:RegularResolvent} to the resolvent operator of problem~\eqref{r1} without any regularisations. This is achieved by constructing an appropriate mollification of the function $\varphi$. In particular, we show that the operator $A$ is $G(0)$-accretive and that the range of $I+\tau A$ is the entire space $L_1$ for $\tau>0$ sufficiently small.
This guarantees that the equation $u_t = Au$ has a unique mild solution for each given initial condition $u^0$. Finally, in Section~\ref{sec:Stability} we prove the main stability result for the mild solutions of equation~\eqref{r1} and then deduce Theorem~\ref{thm:MainIntro} as a special case. The main theorem is proved by introducing the implicit time discretisation, as mentioned above, and thus approximating the evolution operator by the resolvent operator. We conclude the proof by using the stability result from Section~\ref{sec:RegularResolvent} and passing to the limit with the time discretisation.
Let us point out that this strategy is largely inspired by the similar approach of Cockburn and Gripenberg in~\cite{Grip}, who consider an abstract degenerate diffusion problem as in~\eqref{r1} with a hyperbolic term rather than a reaction term.

\section{The resolvent operator in a regularised setting}
\label{sec:RegularResolvent}
We begin by establishing some crucial properties of the resolvent operator in a regularised setting, i.e., with additional regularity of the function $\varphi$ in~\eqref{r1}.
Since it is enough for our applications, we assume that $\vf$ is increasing so that $A$ is single-valued.

\begin{prop}\label{image}
Suppose that $\vf \in C^{2}(\mathbb{R})$, $\vf(0)=0$ and $\vf' > 0$, and that $G$ satisfies (\ref{A2}) and $p$ satisfies (\ref{A3}). We define the operator $A_{\vf,G,p}:D(A) \subset  L_{1}(\mathbb{R}^{d})\cap L_{\infty}(\mathbb{R}^{d}) \rightarrow L_{1}(\mathbb{R}^{d})\cap L_{\infty}(\mathbb{R}^{d})$ with
\[
D(A) = \{ u \in L_{1}(\mathbb{R}^{d})\cap L_{\infty}(\mathbb{R}^{d}) \cap H^{1}(\mathbb{R}^{d}):  \vf
(u) \in H^{2}(\mathbb{R}^{d}) \}
\]
as
\[
A(u)\equiv A_{\vf,G,p}(u) =  -\Delta \vf(u) - (u)_{+}G(p(u)).
\]
Then for every $\tau < \frac{1}{G(0)}$
\begin{enumerate}
\item \[R(I+\tau A) \supset  L_{1}(\mathbb{R}^{d})\cap L_{\infty}(\mathbb{R}^{d}),\]
\item for every $f\in L_{r}(\mathbb{R}^d)$, $r\in [1,\infty]$,

\[ \norm{((I+\tau A)^{-1}f)_{+}}_{L_{r}(\mathbb{R}^d)} \leq \frac{1}{1-\tau G(0)}\norm{f_{+}}_{L_{r}(\mathbb{R}^d)}, \hd \hd \norm{((I+\tau A)^{-1}f)_{-}}_{L_{r}(\mathbb{R}^d)} \leq \norm{f_{-}}_{L_{r}(\mathbb{R}^d)},
\]
\item and \[ - \norm{f_{-}}_{L_{\infty}(\mathbb{R}^d)} \leq(I+\tau A)^{-1}f \leq \frac{1}{1-\tau G(0)}\norm{f_{+}}_{L_{\infty}(\mathbb{R}^d)} \m{ a.e. on } \hd \mathbb{R}^d.\]
\end{enumerate}
\end{prop}
\begin{proof}
We define an operator $B:H^{1}(\mathbb{R}^d) \rightarrow H^{-1}(\mathbb{R}^d)$ by
\[
B(v) = \vf^{-1}(v) - \tau (\vf^{-1}(v))_{+}G(p(\vf^{-1}(v))) - \Delta v.
\]
Without loss of generality, we may assume that $\vf$, $\vf'$, $\vf''$, $G$, $G'$, $p$ are bounded since we deal only with functions that satisfy some a priori given $L_{\infty}$ bound. Clearly, $B$ is bounded form $H^{1}(\mathbb{R}^d)$ to $H^{-1}(\mathbb{R}^d)$. We will show that $B$  is pseudomonotone and coercive. Then, from \cite[Theorem 2.7, p. 180]{Lions}, we will deduce that $B$ is surjective.

To show that $B$ is pseudomonotone we take $v_j \rightharpoonup v$ in $H^{1}(\mathbb{R}^{d})$ and we show that for every $u \in H^{1}(\mathbb{R}^{d})$
\[
\liminf_{j\to\infty} \left\langle B(v_j),v_j-u \right\rangle \geq  \left\langle B(v),v-u \right\rangle.
\]
Firstly, for every $u \in H^{1}(\mathbb{R}^{d})$, we observe
\begin{align*}
\liminf_{j\to\infty} \left\langle \Delta v_j,v_j - u
\right\rangle = - \liminf_{j\to\infty} \int_{\mathbb{R}^d} \nabla v_j \cdot \nabla(v_j-u)dx & = - \liminf_{j\to\infty} \norm{\nabla v_j}_{L_{2}(\mathbb{R}^{d})}^{2} + \lim_{j\to\infty} \int_{\mathbb{R}^d} \nabla v_j \cdot \nabla u dx\\
&\leq -\norm{\nabla v}_{L_{2}(\mathbb{R}^{d})}^{2} + \int_{\mathbb{R}^d} \nabla v \cdot \nabla u dx = \left\langle \Delta v,v - u
\right\rangle,
\end{align*}
where we applied weak lower semicontinuity of the norm.
Secondly, we will show that
\begin{equation}\label{eq:liminf_j}
\liminf_{j\to\infty} \left\langle \vf^{-1}(v_j) - (\vf^{-1}(v_j))_{+}\tau G(p(\vf^{-1}(v_j))),v_j \right\rangle \geq  \left\langle \vf^{-1}(v) - ( \vf^{-1}(v))_{+}\tau G(p(\vf^{-1}(v))),v \right\rangle,
\end{equation}
and, for every $u \in H^{1}(\mathbb{R}^d)$,
\begin{equation}\label{eq:pseudomonotone1}
\lim_{j\to\infty} \left\langle \vf^{-1}(v_j) - (\vf^{-1}(v_j))_{+}\tau G(p(\vf^{-1}(v_j))),u \right\rangle =  \left\langle \vf^{-1}(v) - ( \vf^{-1}(v))_{+}\tau G(p(\vf^{-1}(v))),u \right\rangle.
\end{equation}
Let us denote $H(t) := \vf^{-1}(t) - (\vf^{-1}(t))_{+}\tau G(p(\vf^{-1})(t))$. Then $H(t)t \geq 0$ for every $t \in \mathbb{R}$.
Since $v_j \rightarrow v$ in $L_{2,\mathrm{loc}}(\mathbb{R}^d)$, there exists a subsequence such that $v_{j_k} \rightarrow v$ almost everywhere. Hence, by Fatou's lemma $\lim\inf_{k\to\infty} \left\langle H(v_{j_k}),v_{j_k} \right\rangle \geq \left\langle H(v),v \right\rangle $. We claim that this inequality holds for the entire sequence. Indeed, suppose that there exists another subsequence $j_n$ such that $\liminf_{n\to\infty} \left\langle H(v_{j_n}),v_{j_n} \right\rangle < \left\langle H(v),v \right\rangle$. But since $\norm{v_{j_n}}_{H^{1}(\mathbb{R}^d)}$ is bounded, there exists a further subsequence $j_{n_{l}}$ such that $v_{j_{n_{l}}}\rightarrow v$ a.e and we can apply Fatou's lemma again, which leads to a contradiction. Thus, $\liminf_{j\to\infty} \left\langle H(v_{j}),v_{j} \right\rangle \geq \left\langle H(v),v \right\rangle$, so \eqref{eq:liminf_j} holds true.
To show equality~\eqref{eq:pseudomonotone1}, we note that
\[
\abs{H(v_j)} \leq |\vf^{-1}(v_j)|(1+\tau G(0)) \leq 2|\vf^{-1}(v_j)| \leq 2\norm{(\vf^{-1})'}_{L_{\infty}(\mathbb{R})}|v_j|.
\]
Hence, $\{H(v_j)\}$ is bounded in $L_{2}(\mathbb{R}^d)$, thus there exists a subsequence  $v_{j_k} \to v $ a.e. in $\mathbb{R}^d$,  such that, by continuity of $H$,  $H(v_{j_k}) \rightharpoonup H(v)$ in $L_{2}(\mathbb{R}^{d})$. Since, in this way for every subsequence of $\{H(v_j)\}$ we may choose a subsequence weakly in $L_2(\mathbb{R}^d)$ convergent to $H(v)$, we  infer that
$\lim_{j\to\infty}\left\langle H(v_{j}),u \right\rangle = \left\langle H(v),u \right\rangle $ for every $u \in H^{1}(\mathbb{R}^d)$.

To show that $B$ is coercive we note that due to $\vf^{-1}(t)t \geq 0$ and $\vf^{-1}(0) = 0$, we have
\begin{align*}
    \left\langle H(v),v\right\rangle &= \left\langle \vf^{-1}(v) - ( \vf^{-1}(v))_{+}\tau G(p(\vf^{-1}(v))),v \right\rangle\\
     &= \int_{\mathbb{R}^d}  \vf^{-1}(v) vdx - \tau \int_{\mathbb{R}^d}  ( \vf^{-1}(v))_{+} G(p(\vf^{-1}(v))) (v)_{+} dx\\
    &\geq \int_{\mathbb{R}^d}  \vf^{-1}(v) vdx - \tau G(0)\int_{\mathbb{R}^d}  ( \vf^{-1}(v))_{+} (v)_{+} dx\\
    &\geq (1-\tau G(0))\int_{\mathbb{R}^d}  \vf^{-1}(v) vdx \geq \frac{(1-\tau G(0))}{\norm{\vf'}_{L_{\infty}(\mathbb{R})}} \norm{v}_{L_{2}(\mathbb{R}^{d})}^{2}.
\end{align*}

Hence, by \cite[Theorem 2.7]{Lions}, for every $f \in H^{-1}(\mathbb{R}^d)$ there exists exactly one $v \in H^{1}(\mathbb{R}^{d})$ such that $B(v) = f$. Now suppose $f\in L_1(\R^d)\cap L_\infty(\R^d)$ and let us show that if $u = \vf^{-1}(v)$, then for every $s \in [1,\infty]$
\[
\norm{u_{+}}_{L_{s}(\mathbb{R}^{d})} \leq \frac{1}{1-\tau G(0)}\norm{f_{+}}_{L_{s}(\mathbb{R}^{d})}, \hd \hd \norm{u_{-}}_{L_{s}(\mathbb{R}^{d})} \leq \norm{f_{-}}_{L_{s}(\mathbb{R}^{d})}.
\]
Let us multiply the equation $B(v) = f$ by $K(v)$, where $K$ is a nondecreasing, compactly supported smooth function such that $K(0)=0$. Then $K(v) \in H^{1}(\mathbb{R}^d)$ and we may integrate by parts to obtain
\[
\int_{\mathbb{R}^d} [\vf^{-1}(v)-\tau (\vf^{-1}(v))_{+} G(p(\vf^{-1}(v)))]K(v)dx + \tau \int_{\mathbb{R}^d} \abs{\nabla v}^{2}K'(v)dx = \int_{\mathbb{R}^d} fK(v)dx.
\]
Note that the second term on the left-hand side is nonnegative. Now, for $K$ we take elements of a sequence of functions converging to $(\vf^{-1}(\cdot))_{+}^{s-1}$ from below, where $s \in [1,\infty)$. We arrive at
\[
(1-\tau G(0))\int_{\mathbb{R}^d} (\vf^{-1}(v))^{s}_{+}dx\leq \int_{\mathbb{R}^d} f(\vf^{-1}(v))_{+}^{s-1}dx.
\]
Using $\vf^{-1}(v) = u$ and H\"older's inequality we obtain
\[
\norm{u_{+}}_{L_{s}(\mathbb{R}^d)} \leq \frac{1}{(1-\tau G(0))}\norm{f_{+}}_{L_{s}(\mathbb{R}^d)},\quad\forall s\in[1,\infty).
\]
Passing to the limit with $s$ to infinity, we obtain the result for the sup-norm.
To obtain the result for $u_{-}$ we multiply $B(v)=f$ by $-K(v)$, where $K$ is now nonincreasing. Then
\[
\int_{\mathbb{R}^d}[\vf^{-1}(v)-\tau (\vf^{-1}(v))_{+}G(p(\vf^{-1}(v)))](-K(v))dx  -\tau \int_{\mathbb{R}^d} \abs{\nabla v}^{2}K'(v)dx = \int_{\mathbb{R}^d} (-f)K(v)dx.
\]
Choosing for $K(v)$ the functions approximating $(-\vf^{-1}(v))_{+}^{s-1} =(\vf^{-1}(v))_{-}^{s-1}$,
we get
\[
\int_{\mathbb{R}^d}(\vf^{-1}(v))_{-}^{s} dx \leq \int_{\mathbb{R}^d} (-f) (\vf^{-1}(v))_{-}^{s-1}dx,
\]
and thus
\[
\norm{u_{-}}_{L_{s}(\mathbb{R}^d)} \leq \norm{f_{-}}_{L_{s}(\mathbb{R}^d)},\quad\forall s\in[1,\infty),
\]
and again passing to the limit with $s$ we obtain the estimate for the supremum norm.
Hence, we have shown the results 2.\ and 3. To finish the proof of 1., we note that for every $f \in L_{1}(\mathbb{R}^{d}) \cap L_{\infty}(\mathbb{R}^{d})$ there exists exactly one $v \in H^{1}(\mathbb{R}^{d})$ such that $B(v) = f$ and $u = \vf^{-1}(v) \in L_{1}(\mathbb{R}^{d}) \cap L_{\infty}(\mathbb{R}^{d})$. In particular, $u \in L_{2}(\mathbb{R}^{d})$ and since
\[
u -\tau \Delta \vf(u) - \tau (u_{+})G(p(u)) = f,
\]
we have $\Delta \vf(u) \in L_{2}(\mathbb{R}^{d})$. Finally, $|\vf(u)| \leq \norm{\vf'}_{L_{\infty}(\mathbb{R}^{d})}|u|$ implies $\vf(u) \in  L_{1}(\mathbb{R}^{d}) \cap L_{\infty}(\mathbb{R}^{d})$ and $\nabla u = (\vf^{-1}(v))' \nabla v$ leads to $u \in H^{1}(\mathbb{R}^{d})$.
\end{proof}

The next result shows that $A$ is an accretive operator and establishes a first stability estimate in the regularised setting. We denote by $q$ a smooth, nonnegative, and even function on $\mathbb{R}$ with support contained in $[-1,1]$, such that $q(0) = 1$ and $\int_{\mathbb{R}}qdx = 1$. Then we set
\[
q^{\ve}(x) = \frac{1}{\ve^{d}}\prod_{i=1}^{d}q(\frac{x_i}{\ve}), \m{ where } x = (x_1,\dots,x_d).
\]

\begin{prop}\label{acr}
Let $f_1,f_2 \in L_{1}(\mathbb{R}^d)\cap L_{\infty}(\mathbb{R}^d)$ and $\tau \in (0,\frac{1}{G(0)})$. Under the assumptions of Proposition~\ref{image}, there holds
    \begin{enumerate}
    \item  \eqq{
    \norm{(I+\tau A)^{-1}f_{1} - (I+\tau A)^{-1}f_{2}}_{ L_{1}(\mathbb{R}^d)} \leq \frac{1}{1-\tau G(0)}\norm{f_{1}-f_2}_{L_{1}(\mathbb{R}^d)},
    }{acretive}
    \item \eqq{\norm{(I+\tau A)^{-1}f}_{TV(\mathbb{R}^{d})} \leq \frac{1}{1-\tau G(0)}\norm{f}_{TV(\mathbb{R}^{d})} \m{ for any } f\in BV(\mathbb{R}^{d})\cap L_{1}(\mathbb{R}^d)\cap L_{\infty}(\mathbb{R}^d),}{tv}

\item for $i=1,2$ let $G_i,p_i$ satisfy (\ref{A2}), (\ref{A3}), respectively and $\vf_i$ satisfy the assumptions of Proposition \ref{image}, and let $A_i =A_{\vf_i,G_i,p_i}$. Define $G(t)=\max\{G_{1}(t),G_{2}(t)\}$ and assume additionally that $f_1$ belongs to $BV(\mathbb{R}^d)$. Then, for
\eqq{
u_1 := (I+\tau A_1)^{-1}f_1, \hd u_2 := (I+\tau A_2)^{-1}f_2,
}{aa0}
 we have
\begin{equation}\label{estepsilon}
\begin{aligned}
&\int_{\mathbb{R}^{d}}\int_{\mathbb{R}^{d}}\abs{u_1(x)-u_2(y)}q^{\ve}(x-y)dydx\\[0.5em]
&\;\leq \frac{1}{(1-\tau G(0))^{2}}
\frac{2 d \tau}{\ve} \sup_{s\in I(f_{1})}\left[\sqrt{\vf'_{1}(s)} - \sqrt{\vf'_{2}(s)}\right]^2\norm{f_1}_{TV(\mathbb{R}^{d})}\\[0.5em]
&\;\;+ \frac{1}{1-\tau G(0)}\int_{\mathbb{R}^{d}}\int_{\mathbb{R}^{d}}\abs{f_1(x)-f_2(y)}q^{\ve}(x-y)dydx \\[0.5em]
&\;\;+\frac{\tau}{(1-\tau G(0))^{2}} \left(\norm{G'_2}_{L_{\infty}(I_{p_1,p_2}(f_1))}\sup_{s \in I(f_1)}\abs{p_{2}(s)-p_{1}(s)}+\norm{G_2-G_1}_{L_{\infty}(I_{p_1}(f_1))}\right)\norm{f_1}_{L_{1}(\mathbb{R}^{d})},
\end{aligned}
\end{equation}
 \end{enumerate}
where we denote
    \eqq{
    I(f_1) := \left[-\norm{(f_1)_{-}}_{L_{\infty}(\mathbb{R}^{d})},\frac{1}{1-\tau G(0)}\norm{(f_{1})_+}_{L_{\infty}(\mathbb{R}^{d})}\right],
    }{If1}
    \[
    I_{p_1,p_2}(f_1) := \left[\min_{i\in \{1,2\}} p_i\left(-\norm{(f_{1})_{-}}_{L_{\infty}(\mathbb{R}^d)}
\right),\max_{i\in \{1,2\}}p_i\left(\frac{\norm{(f_{1})_+}_{L_{\infty}(\mathbb{R}^d)}}{1-\tau G(0)}\right) \right],
    \]
    \[
    I_{p_i}(f_1) := \left[  p_i\left(-\norm{(f_{1})_{-}}_{L_{\infty}(\mathbb{R}^d)}\right),  p_i\left(\frac{\norm{(f_{1})_+}_{L_{\infty}(\mathbb{R}^d)}}{1-\tau G(0)}\right)\right], \hd i = 1,2,
    \]
\[
\norm{f}_{TV(\mathbb{R}^{d})} = \sup_{\Phi \in C_{0}^{1}(\mathbb{R}^{d};\mathbb{R}^{d}), \norm{\Phi}_{L_{\infty}(\mathbb{R}^{d})} = 1} \int_{\mathbb{R}^{d}}f \divv  \Phi dx
\]
and by $BV(\mathbb{R}^{d})$ we denote the set of all locally integrable  functions with finite $TV$ - norm.
\end{prop}

\begin{proof}

The proof is an adaptation of the ideas in \cite{Grip} and follows the classical doubling of variables technique.
We define
$S_{\eta}(t) = S(t/\eta), \eta > 0$, where $S$ is an odd, nondecreasing, smooth approximation of the signum function.
Let us fix $\tau \in
(0,\frac{1}{G(0)})$ and take $f_i \in L_{1}(\mathbb{R}^d)\cap L_{\infty}(\mathbb{R}^d)$  for $i=1,2$. Then, due to Proposition~\ref{image}, there exist $u_1, u_2 \in D(A_1) = D(A_2)$, i.e. $u_1,u_2 \in  L_{1}(\mathbb{R}^{d})\cap L_{\infty}(\mathbb{R}^{d})\cap H^{1}(\mathbb{R}^{d})$ satisfying for $i=1,2 \hd \hd \Delta \vf(u_i) \in L_{2}(\mathbb{R}^{d})$ and
\eqq{
(I+\tau A_1)u_1 = f_1, \hd (I+\tau A_2)u_2 = f_2.
}{aa}
We apply the method of doubling the variables. To this end, we will consider $u_1,f_1$ as functions of $x \in \mathbb{R}^{d}$ and $u_2,f_2$ as functions of $y \in \mathbb{R}^{d}$. Moreover, we introduce
\begin{equation}\label{eq:Fetai}
F_{i}^{\eta}(s,t) = \int_{t}^{s}\vf'_{i}(\sigma)S_{\eta}(\sigma-t)d\sigma, \hd F_{1,2}^{\eta}(s,t) = \int_{s}^{t}\sqrt{\vf_{2}(r)}\int_{r}^{s}\sqrt{\vf_{1}(\sigma)}S'_{\eta}(\sigma-r)d\sigma dr.
\end{equation}
Then, clearly
\[
\Delta_{x}(\vf_{1}(u_1))S_{\eta}(u_1-u_2) = \Delta_{x}F_{1}^{\eta}(u_1,u_2) - \vf'_{1}(u_1)S'_{\eta}(u_1-u_2)\abs{\nabla_x u_1}^{2},
\]
\[
\Delta_{y}(\vf_{2}(u_2))S_{\eta}(u_2-u_1) = \Delta_{y}F_{2}^{\eta}(u_2,u_1) - \vf'_{2}(u_2)S'_{\eta}(u_2-u_1)\abs{\nabla_y u_2}^{2},
\]
\[
\nabla_x \nabla_y F_{1,2}^{\eta}(u_1,u_2) = \sqrt{\vf'_{1}(u_1)}\sqrt{\vf'_{2}(u_2)}S'_{\eta}(u_1-u_2)\nabla_{x}u_1 \cdot \nabla_{y} u_2.
\]
We now multiply the first equation in (\ref{aa}) by $S_{\eta}(u_1-u_2)$, the second one by $S_{\eta}(u_2-u_1)$ and add them together to the result
\[
(u_1-u_2)S_{\eta}(u_1-u_2) - \tau\Delta_x(\vf_{1}(u_1))S_{\eta}(u_1-u_2) -\tau\Delta_y(\vf_{2}(u_2))S_{\eta}(u_2-u_1)
\]
\eqq{
\;\;+ \tau[(u_2)_{+}G_{2}(p_2(u_2)) - (u_1)_{+}G_{1}(p_1(u_1))] S_{\eta}(u_1-u_2) = (f_1-f_2)S_{\eta}(u_1-u_2).
}{ab}
Since $S'_{\eta}$ is nonnegative and even, we have
\begin{align*}
    -&\Delta_x(\vf_{1}(u_1))S_{\eta}(u_1-u_2) -\Delta_y(\vf_{2}(u_2))S_{\eta}(u_2-u_1)\\[0.5em]
    &= -\Delta_{x}F_{1}^{\eta}(u_1,u_2) - \Delta_{y}F_{2}^{\eta}(u_2,u_1) + \left[\vf'_{1}(u_1)\abs{\nabla_x u_1}^{2}+ \vf'_{2}(u_2)\abs{\nabla_y u_2}^{2}\right]  S'_{\eta}(u_1-u_2)\\[0.5em]
    &\geq  -\Delta_{x}F_{1}^{\eta}(u_1,u_2) - \Delta_{y}F_{2}^{\eta}(u_2,u_1) + 2\sqrt{\vf'_{1}(u_1)}\sqrt{\vf'_{2}(u_2)}S'_{\eta}(u_1-u_2)\nabla_{x}u_1 \cdot \nabla_{y} u_2\\[0.5em]
    &=-\Delta_{x}F_{1}^{\eta}(u_1,u_2) - \Delta_{y}F_{2}^{\eta}(u_2,u_1) + 2\nabla_x \nabla_y F_{1,2}^{\eta}(u_1,u_2).
\end{align*}
Inserting this into (\ref{ab}) we arrive at
\[
(u_1-u_2)S_{\eta}(u_1-u_2)-\tau\Delta_{x}F_{1}^{\eta}(u_1,u_2) - \tau\Delta_{y}F_{2}^{\eta}(u_2,u_1) + 2\tau \nabla_x \nabla_y F_{1,2}^{\eta}(u_1,u_2)
\]
\eqq{
\;\;+ \tau[(u_2)_{+}G_{2}(p_2(u_2)) - (u_1)_{+}G_{1}(p_1(u_1))] S_{\eta}(u_1-u_2) \leq (f_1-f_2)S_{\eta}(u_1-u_2).
}{ac}
Let $\Psi$ be a smooth function with compact support in $\mathbb{R}^{d}$ such that $\Psi\equiv 1$ for $\abs{x} \leq 1$ and $0\leq \Psi \leq 1$. We multiply both sides of (\ref{ac}) by the nonnegative function $\Phi^{\ve}_m(x,y) = q^{\ve}(x-y)\Psi(x/m)\Psi(y/m)$, $m \in \mathbb{N}$ and integrate in $x$ and $y$, obtaining
\begin{align*}
    \int_{\mathbb{R}^{d}}\int_{\mathbb{R}^{d}}&\left[(u_1-u_2)S_{\eta}(u_1-u_2)-\tau\Delta_{x}F_{1}^{\eta}(u_1,u_2) - \tau\Delta_{y}F_{2}^{\eta}(u_2,u_1) + 2\tau \nabla_x \nabla_y F_{1,2}^{\eta}(u_1,u_2) \right]\Phi^{\ve}_m(x,y)dydx\\
    &+\int_{\mathbb{R}^{d}}\int_{\mathbb{R}^{d}} \tau[(u_2)_{+}G_{2}(p_2(u_2)) - (u_1)_{+}G_{1}(p_1(u_1))] S_{\eta}(u_1-u_2)\Phi^{\ve}_m(x,y)dydx\\ &\leq \int_{\mathbb{R}^{d}}\int_{\mathbb{R}^{d}}(f_1-f_2)S_{\eta}(u_1-u_2)\Phi^{\ve}_m(x,y)dydx.
\end{align*}
Since $\Delta_x q^{\ve}(x-y) = \Delta_y q^{\ve}(x-y)$ and $\nabla_x q^{\ve}(x-y) = - \nabla_{y}q^{\ve}(x-y)$, integration by parts gives
\begin{align*}
    -&\int_{\mathbb{R}^{d}}\int_{\mathbb{R}^{d}}\left[\Delta_{x}F_{1}^{\eta}(u_1,u_2) - \Delta_{y}F_{2}^{\eta}(u_2,u_1) + 2\nabla_x \nabla_y F_{1,2}^{\eta}(u_1,u_2) \right] \Phi^{\ve}_m(x,y)dydx\\
    &=-\int_{\mathbb{R}^{d}}\int_{\mathbb{R}^{d}}\left[F_{1}^{\eta}(u_1,u_2) + F_{2}^{\eta}(u_2,u_1)  + 2  F^{\eta}_{1,2}(u_1,u_2) \right]\Delta_{x}q^{\ve}(x-y)\Psi(x/m)\Psi(y/m)dydx\\ &\quad-\int_{\mathbb{R}^{d}}\int_{\mathbb{R}^{d}} E_{m}(x,y)dydx,
\end{align*}
where
\[
 E_{m}(x,y) = F^{\eta}_1(u_1,u_2)\left[q^{\ve}(x-y)\Delta_x \Psi(x/m)\Psi(y/m)m^{-2} + 2\nabla_x q^{\ve}(x-y) \cdot \nabla\Psi(x/m)\Psi(y/m)m^{-1}\right]
\]
\[
+F^{\eta}_2(u_2,u_1)\left[q^{\ve}(x-y)\Delta_y \Psi(y/m)\Psi(x/m)m^{-2} - 2\nabla_x q^{\ve}(x-y) \cdot \nabla\Psi(y/m)\Psi(x/m)m^{-1}\right]
\]
\[
-2 F^{\eta}_{1,2}(u_1,u_2)\left[q^{\ve}(x-y)\nabla_x \Psi(x/m)\nabla_y\Psi(y/m)m^{-2} + \nabla_x q^{\ve}(x-y) \cdot \left(\Psi(x/m)\nabla\Psi(y/m)-\nabla \Psi(x/m) \Psi(y/m)\right)m^{-1}\right].
\]
We note that by \eqref{eq:Fetai}
\[
|F^{\eta}_1(u_1,u_2)| + |F^{\eta}_2(u_2,u_1)| + |F^{\eta}_{1,2}(u_1,u_2)| \leq C(|u_1| + |u_2|),
\]
for some nonnegative $C=C(\norm{\vf_{1}}_{W^{1,\infty}(I(f_1)\cup I(f_2))}, \norm{\vf_2}_{W^{1,\infty}(I(f_1)\cup I(f_2))})$,
and
\[
\abs{(u_2)_{+}G_{2}(p_2(u_2)) - (u_1)_{+}G_{1}(p_1(u_1))} \leq (\norm{G_1}_{L_{\infty}(\mathbb{R})} +\norm{G_2}_{L_{\infty}(\mathbb{R})}) (|u_2| + |u_1|).
\]
Since $|u_1(x)| + |u_2(y)|$ is integrable on $\{(x,y): \norm{x-y}_{\infty} \leq \ve \}$, by the Lebesgue dominated convergence theorem we may pass to the limit with $m$ and we arrive at
\[
\int_{\mathbb{R}^{d}}\int_{\mathbb{R}^{d}}(u_1-u_2)S_{\eta}(u_1-u_2)q^{\ve}(x-y)dydx-\tau\int_{\mathbb{R}^{d}}\int_{\mathbb{R}^{d}}\left[F_{1}^{\eta}(u_1,u_2) + F_{2}^{\eta}(u_2,u_1)  + 2  F^{\eta}_{1,2}(u_1,u_2) \right]\Delta_{x}q^{\ve}(x-y)dydx
\]
\[
+\tau\int_{\mathbb{R}^{d}}\int_{\mathbb{R}^{d}} [(u_2)_{+}G_{2}(p_2(u_2)) - (u_1)_{+}G_{1}(p_1(u_1))] S_{\eta}(u_1-u_2)q^{\ve}(x-y)dydx 
\]
\[
\leq \int_{\mathbb{R}^{d}}\int_{\mathbb{R}^{d}}(f_1-f_2)S_{\eta}(u_1-u_2)q^{\ve}(x-y)dydx.
\]
Integrating by parts again
\begin{equation}\label{ad}
    \begin{aligned}
        &\int_{\mathbb{R}^{d}}\int_{\mathbb{R}^{d}}(u_1-u_2)S_{\eta}(u_1-u_2)q^{\ve}(x-y)dydx\\[0.5em]
        &\quad+\tau\int_{\mathbb{R}^{d}}\int_{\mathbb{R}^{d}}\left[\partial_{s}F_{1}^{\eta}(s,t) + \partial_{s}F_{2}^{\eta}(t,s)  + 2 \partial_{s} F^{\eta}_{1,2}(s,t) \right]|_{s=u_{1}(x), t = u_{2}(y)} \nabla_x u_1 \cdot \nabla_{x}q^{\ve}(x-y)dydx\\[0.5em]
        &\quad +\tau\int_{\mathbb{R}^{d}}\int_{\mathbb{R}^{d}} [(u_2)_{+}G_{2}(p_2(u_2)) - (u_1)_{+}G_{1}(p_1(u_1))] S_{\eta}(u_1-u_2)q^{\ve}(x-y)dydx\\[0.5em]
        &\leq \int_{\mathbb{R}^{d}}\int_{\mathbb{R}^{d}}(f_1-f_2)S_{\eta}(u_1-u_2)q^{\ve}(x-y)dydx.
    \end{aligned}
\end{equation}
We observe that
\begin{align*}
    \partial_{s}F_{1}^{\eta}(s,t) &+ \partial_{s}F_{2}^{\eta}(t,s)  + 2 \partial_{s} F^{\eta}_{1,2}(s,t)\\
    &=\vf'_{1}(s)S_{\eta}(s-t) - \int_{s}^{t}\vf'_{2}(\sigma)S'_{\eta}(\sigma-s)d\sigma + 2\int_{s}^{t}\sqrt{\vf_{2}(\sigma)}{\sqrt{\vf_{1}(s)}}S'_{\eta}(s-\sigma)d\sigma\\
    &= \int_{s}^{t}S'_{\eta}(s-\sigma)\left[\sqrt{\vf'_{1}(s)} - \sqrt{\vf'_{2}(\sigma)}\right]^2 d\sigma.
\end{align*}
We may choose $S_\eta$ such that $|S'_\eta| \leq 2/\eta$ and $\supp (S'_\eta(t)) \subset (-\eta,\eta)$.
Hence, by the Lebesgue differentiation theorem for every $t$ and $s$ we have
\eqq{
\partial_{s}F_{1}^{\eta}(s,t) + \partial_{s}F_{2}^{\eta}(t,s)  + 2 \partial_{s} F^{\eta}_{1,2}(s,t) \rightarrow \mathrm{sign}(s-t)\left[\sqrt{\vf'_{1}(s)} - \sqrt{\vf'_{2}(s)}\right]^2.
}{ae}
Moreover, the expression above may be bounded independently on $\eta$ 
\[
\abs{\int_{s}^{t}S'_{\eta}(s-\sigma)\left[\sqrt{\vf'_{1}(s)} - \sqrt{\vf'_{2}(\sigma)}\right]^2 d\sigma} \leq \sup_{s}\left[\sqrt{\vf'_{1}(s)} - \sqrt{\vf'_{2}(s)}\right]^2\frac{2}{\eta}\min\{\eta,|t-s|\} 
\]
\[
\leq 2\sup_{s}\left[\sqrt{\vf'_{1}(s)} - \sqrt{\vf'_{2}(s)}\right]^2.
\]
Furthermore,
\[
\int_{\mathbb{R}^{d}}\int_{\mathbb{R}^{d}} [(u_2)_{+}G_{2}(p_2(u_2)) - (u_1)_{+}G_{1}(p_1(u_1))] S_{\eta}(u_1-u_2)q^{\ve}(x-y)dydx
\]
\[
= \int_{\mathbb{R}^{d}}\int_{\mathbb{R}^{d}} [(u_2)_{+}G_{2}(p_2(u_2)) - (u_1)_{+}G_{2}(p_1(u_1))] S_{\eta}(u_1-u_2)q^{\ve}(x-y)dydx
\]
\eqq{
+ \int_{\mathbb{R}^{d}}\int_{\mathbb{R}^{d}} [G_{2}(p_1(u_1)) - G_{1}(p_1(u_1))](u_1)_{+} S_{\eta}(u_1-u_2)q^{\ve}(x-y)dydx.
}{af}
Let us denote for $f \in L_{\infty}(\mathbb{R}^d)$
\[
J_{p_1,p_2}(f) := \left[\min_{i\in \{1,2\}}\inf_{x\in \mathbb{R}^d}p_i(f(x)),\max_{i\in \{1,2\}}\sup_{x\in \mathbb{R}^d}p_i(f(x)) \right], \hd  J_{p_i}(f) := \left[ \inf_{x \in \mathbb{R}^d} p_i(f(x)), \sup_{x \in \mathbb{R}^d} p_i(f(x))\right], \hd i = 1,2.
\]
Then, the second term of~\eqref{af} may be estimated as follows
\[
\abs{\int_{\mathbb{R}^{d}}\int_{\mathbb{R}^{d}} [G_{2}(p_1(u_1)) - G_{1}(p_1(u_1))](u_1)_{+} S_{\eta}(u_1-u_2)q^{\ve}(x-y)dydx}
\]
\eqq{
\leq \norm{G_{2}-G_{1}}_{L_{\infty}(J_{p_1}(u_1))}\int_{\mathbb{R}^{d}}\int_{\mathbb{R}^{d}}(u_1)_{+} q^{\ve}(x-y)dydx,
}{ag}
while the first one may be decomposed into three parts
\begin{equation}\label{ah}
    \begin{aligned}
        \int_{\mathbb{R}^{d}}\int_{\mathbb{R}^{d}} &[(u_2)_{+}G_{2}(p_2(u_2)) - (u_1)_{+}G_{2}(p_1(u_1))] S_{\eta}(u_1-u_2)q^{\ve}(x-y)dydx \\
        =& \int_{\mathbb{R}^{d}}\int_{\mathbb{R}^{d}} (u_2)_{+}[G_{2}(p_2(u_2)) - G_{2}(p_2(u_1))] S_{\eta}(u_1-u_2)q^{\ve}(x-y)dydx\\
        &+ \int_{\mathbb{R}^{d}}\int_{\mathbb{R}^{d}} G_{2}(p_2(u_1))((u_2)_{+}-(u_1)_{+}) S_{\eta}(u_1-u_2)q^{\ve}(x-y)dydx \\
        &+\int_{\mathbb{R}^{d}}\int_{\mathbb{R}^{d}} (u_1)_{+}[G_{2}(p_2(u_1)) - G_{2}(p_1(u_1))] S_{\eta}(u_1-u_2)q^{\ve}(x-y)dydx \equiv I_{1}+I_2+I_3.
    \end{aligned}
\end{equation}
Since $p_2$ is nondecreasing and nonnegative and $G_2$ is nonincreasing, we infer that $I_1 \geq 0$.
Moreover,
\eqq{
\abs{I_2} \leq G(0)\int_{\mathbb{R}^{d}}\int_{\mathbb{R}^{d}}\abs{u_1-u_2}q^{\ve}(x-y)dydx,
}{ai}
\eqq{
\abs{I_3} \leq \norm{G'_2}_{L_{\infty}(J_{p_1,p_2}(u_1))} \sup_{s \in I(f_1)}\abs{p_{2}(s)-p_{1}(s)}\int_{\mathbb{R}^{d}}\int_{\mathbb{R}^{d}}(u_1)_{+}q^{\ve}(x-y)dydx,
}{aj}
where $I(f_1)$ is defined in (\ref{If1}).
Using (\ref{ae})--(\ref{aj}) in (\ref{ad}) and passing to the limit $\eta\to0$ in (\ref{ad}) leads to
\[
\int_{\mathbb{R}^{d}}\int_{\mathbb{R}^{d}}\abs{u_1-u_2}q^{\ve}(x-y)dydx
+\tau\int_{\mathbb{R}^{d}}\int_{\mathbb{R}^{d}}\left[\sqrt{\vf'_{1}(u_1)} - \sqrt{\vf'_{2}(u_1)}\right]^2 \mathrm{sign}(u_2-u_1) \nabla_x u_1 \cdot \nabla_{x}q^{\ve}(x-y)dydx
\]
\[
\leq \int_{\mathbb{R}^{d}}\int_{\mathbb{R}^{d}}\abs{f_1-f_2}q^{\ve}(x-y)dydx + \tau G(0)\int_{\mathbb{R}^{d}}\int_{\mathbb{R}^{d}}\abs{u_1-u_2}q^{\ve}(x-y)dydx
\]
\eqq{
+\tau \left(\norm{G'_2}_{L_{\infty}(J_{p_1,p_2}(u_1))}\sup_{s \in I(f_1)}\abs{p_{2}(s)-p_{1}(s)}+\norm{G_2-G_1}_{L_{\infty}(J_{p_1}(u_1))}\right)\int_{\mathbb{R}^{d}}\int_{\mathbb{R}^{d}}|u_1|q^{\ve}(x-y)dydx .
}{ak}
Hence, in the case $\vf_1=\vf_2$, $G_1=G_2$, $p_1=p_2$, passing to the limit with $\ve\to0$ gives
\[
\norm{u_{1}-u_2}_{L_{1}(\mathbb{R}^{d})} \leq \frac{1}{1-\tau G(0)} \norm{f_{1}-f_2}_{L_{1}(\mathbb{R}^{d})},
\]
and thus (\ref{acretive}) is proven. Estimate (\ref{tv}) follows from  (\ref{acretive}). Indeed, denoting the translation operator by $T_y$, i.e., $(T_yf)(x) = f(x+y)$, we have
\[
\norm{T_y(I+\tau A)^{-1}f - (I+\tau A)^{-1}f}_{L_{1}(\mathbb{R}^{d})} = \norm{(I+\tau A)^{-1}T_yf - (I+\tau A)^{-1}f}_{L_{1}(\mathbb{R}^{d})} 
\]
\[
\leq \frac{1}{1-\tau G(0)}\norm{T_yf-f}_{L_{1}(\mathbb{R}^{d})},
\]
and (\ref{tv}) follows from the characterisation of the $TV$-norm.
Furthermore, from the properties of $q^{\ve}$ we have
\[
\int_{\mathbb{R}^{d}}q^{\ve}(x-y)dy = 1, \hd \text{and} \hd \int_{\mathbb{R}^{d}}\abs{ q^{\ve}_{x_i}(x-y)}dy = \frac{2}{\varepsilon}, \hd i = 1,\dots,d.
\]
Since (\ref{tv}) implies that if $ f_1 \in BV(\mathbb{R}^{d})$, then also $u_1 \in  BV(\mathbb{R}^{d})$, the inequality (\ref{ak}) for  $ f_1 \in BV(\mathbb{R}^{d})$ leads to
\[
(1-\tau G(0))\int_{\mathbb{R}^{d}}\int_{\mathbb{R}^{d}}\abs{u_1-u_2}q^{\ve}(x-y)dydx 
\]
\[
\leq
\frac{2 d \tau}{\ve} \sup_{s\in I(f_{1})}\left[\sqrt{\vf'_{1}(s)} - \sqrt{\vf'_{2}(s)}\right]^2\norm{u_1}_{TV(\mathbb{R}^{d})}
+ \int_{\mathbb{R}^{d}}\int_{\mathbb{R}^{d}}\abs{f_1-f_2}q^{\ve}(x-y)dydx 
\]
\[
+ \tau \left(\norm{G'_2}_{L_{\infty}(J_{p_1,p_2}(u_1))}\sup_{s \in I(f_1)}\abs{p_{2}(s)-p_{1}(s)}+\norm{G_2-G_1}_{L_{\infty}(J_{p_1}(u_1))}\right)\norm{u_1}_{L_{1}(\mathbb{R}^{d})}.
\]
Now, we extend the intervals $J_{p_1}, J_{p_1,p_2}$ to $I_{p_1}, I_{p_1,p_2}$  applying Proposition \ref{image}, and the fact that $p_i$ are nondecreasing. Finally, due to estimate (\ref{tv}), we arrive at (\ref{estepsilon}).
\end{proof}

\section{Existence and approximability of solutions to the resolvent equation}\label{sec:Resolvent}
In this section we show that the solution to resolvent equation $u = (I+\tau A)^{-1}f$, where $f \in L_{1}(\mathbb{R}^d)$ and $Au = -\Delta \vf(u)-(u)_{+}G(p(u))$ with $\vf,G,p$ satisfying merely (\ref{A1}) - (\ref{A3}) may be nicely approximated by solutions to the regularized problem $u_n = (1+\tau A_n)^{-1}f_n$ and the results established in Proposition~\ref{acr} for a regularized setting may be transferred to our initial setting. Let us begin with the following $L_1$-solvability result, which fits within the scope of the classical paper \cite{BBC}.
\begin{prop}\label{exi}
Let $\vf$, $G$, $p$ satisfy (\ref{A1})-(\ref{A3}) and let $
\tau \in (0,\frac{1}{G(0)})$. For every $f \in L_{1}(\mathbb{R}^{d})$ there exists exactly one $u \in L_{1}(\mathbb{R}^{d})$ which satisfies
\eqq{
u -\tau \Delta \vf(u) - \tau (u)_{+}G(p(u)) = f,
}{Pw}
in the sense of distributions. Furthermore,
\[
\norm{\vf(u)}_{L_{1}(\mathbb{R}^d)}  \leq C\norm{f}_{L_{1}(\mathbb{R}^d)}, \hd
\norm{u}_{L_{1}(\mathbb{R}^d)} \leq \frac{1}{1-\tau G(0)}\norm{f}_{L_{1}(\mathbb{R}^d)}.
\]
\end{prop}
\begin{proof}
Let $\vf$ satisfy  (\ref{A1}) and $\tau \in (0,\frac{1}{G(0)})$ be fixed. We discuss the problem
\eqq{
-\tau \Delta v + H(v) = f,
}{P}
with
$H(r) =  \vf^{-1}(r) - \tau (\vf^{-1}(r))_{+}G(p(\vf^{-1}(r))) \m{ for } r \in \mathbb{R}.$
Clearly, $H(0) = 0$ and $H$ is increasing, since $\vf^{-1}$ is increasing, $G$ is nonincreasing and $p$ is nondecreasing. Furthermore, for $\abs{r} < 1$
\[
\abs{H(r)} \geq (1-\tau G(0))\abs{\vf^{-1}(r)} \geq \frac{(1-\tau G(0))}{\norm{\vf'}_{L_{\infty}([-1,1])}}|r|.
\]
Hence, from \cite[Theorem 2.1, Theorem 5.1]{BBC} we obtain that for every $f \in L_{1}(\mathbb{R}^{d})$ there exists exactly one $v \in L_{1}(\mathbb{R}^d)$ which satisfies (\ref{P}) in the sense of distributions. Furthermore,
\eqq{
\norm{v}_{L_{1}(\mathbb{R}^d)}  \leq C\norm{f}_{L_{1}(\mathbb{R}^d)}, \hd \norm{H(v)}_{L_{1}(\mathbb{R}^d)}\leq \norm{f}_{L_{1}(\mathbb{R}^d)}.
}{ca}
Then, the function $u := \vf^{-1}(v)$ satisfies (\ref{Pw}). Since
\[
\norm{H(v)}_{L_{1}(\mathbb{R}^d)} = \norm{u - \tau (u)_{+}G(p(u))}_{L_{1}(\mathbb{R}^d)},
\]
estimates (\ref{ca}) imply
\[
\norm{\vf(u)}_{L_{1}(\mathbb{R}^d)}  \leq C\norm{f}_{L_{1}(\mathbb{R}^d)}, \hd
\norm{u}_{L_{1}(\mathbb{R}^d)} \leq \frac{1}{1-\tau G(0)}\norm{f}_{L_{1}(\mathbb{R}^d)}.\qedhere
\]
\end{proof}

We will now show that one may approximate the $L_1$ solutions to (\ref{Pw}) by the solutions to a regularized problem.
To this end, for $n\in\mathbb{N}$ let us define
\eqq{
\vf_n(t) = \int_{0}^{t}(\eta_{\frac{1}{n}}*\vf')(s)ds + \frac{t}{n},
}{vfapro}
where $\eta
_{1/n}$ denotes the standard smoothing kernel at scale $1/n$. Then $\vf_n$ satisfies (\ref{A1}) and $\vf_n \in C^{2}(\mathbb{R})$, $\vf_n' > 0$.  Furthermore, $\vf'_n$ converges to $\vf'$ uniformly on compact subsets of $\mathbb{R}$. We denote $A_n:= A_{\vf_{n},G,p}$.
\begin{prop}\label{apro}
    Let $f\in L_1(\R^d)$ and let $u$ be the unique solution of~\eqref{Pw} associated to $f$. There exists a sequence $\{f_n\}\subset L_1(\R^d)\cap L_\infty(\R^d)\cap BV(\R^d)$ such that $f_n\to f$ and $u_n\to u$ in $L_1$, where $u_n := (I-\tau A_{n})^{-1}f_n \in D(A_n)$.

\end{prop}
\begin{proof}
Let us define $f_{m}$ as a truncation of $f$ on the level $m$, i.e. $f_m = f\chi_{\{|f|\leq m\}} + m \sign{f}\chi_{\{|f| > m\}}$. Then $f_{m} \in L_{1}(\mathbb{R}^d) \cap L_{\infty}(\mathbb{R}^d)$. Further, we define $f_{m,n}:=\eta_{1/n}*f_m$. Then
\[
\norm{f_{m,n}}_{L_{\infty}(\mathbb{R}^d)}\leq \norm{ \eta_{1/n}}_{L_{1}(\mathbb{R}^{d})}\norm{f_{m}}_{L_{\infty}(\mathbb{R}^d)}  \leq  m,
\]
\eqq{
\norm{f_{m,n}}_{TV(\mathbb{R}^{d})} \leq \norm{f_{m}}_{L_{1}(\mathbb{R}^{d})}\norm{\nabla \eta_{1/n}}_{L_{1}(\mathbb{R}^{d})} \leq n \norm{\nabla \eta}_{L_{1}(\mathbb{R}^{d})}\norm{f}_{L_{1}(\mathbb{R}^{d})}.
}{fnmbounds}
We will now choose the index $m$ as a function of $n$, i.e, $m = m(n)$, and we denote $f_{m,n}$ as $f_n$.
By Proposition \ref{image} and Proposition \ref{acr}, for every $n \in \mathbb{N}$ there exists a unique $u_n \in L_{1}(\mathbb{R}^d)\cap  L_{\infty}(\mathbb{R}^d)\cap H^{1}(\mathbb{R}^d), \varphi_n(u_n) \in H^{2}(\mathbb{R}^d)$ (note that the uniqueness comes from (\ref{acretive})), such that
\eqq{
u_n -\tau \Delta \vf_n(u_n) - \tau (u_{n})_{+}G(p(u_{n})) = f_n.
}{cb}
We will show that $u_n$ is a Cauchy sequence. Estimate (\ref{estepsilon}) gives
\begin{align*}
    \int_{\mathbb{R}^{d}}\int_{\mathbb{R}^{d}}\abs{u_n-u_k}q^{\ve}(x-y)dydx &\leq \frac{1}{1-\tau G(0)}\int_{\mathbb{R}^{d}}\int_{\mathbb{R}^{d}}\abs{f_n-f_k}q^{\ve}(x-y)dydx\\
    &\quad+\frac{1}{(1-\tau G(0))^{2}}
\frac{2 d \tau}{\ve} \sup_{s\in I(f_{n})}\left[\sqrt{\vf'_{n}(s)} - \sqrt{\vf'_{k}(s)}\right]^2\norm{f_n}_{TV(\mathbb{R}^{d})}.
\end{align*}
Applying estimate (\ref{tv}) and
\[
\abs{\int_{\mathbb{R}^{d}}\int_{\mathbb{R}^{d}}\abs{f(x)-g(y)}q^{\ve}(x-y)dydx - \norm{f-g}_{L_{1}(\mathbb{R}^d)}  } \leq  \ve \norm{f}_{TV(\mathbb{R}^d)},
\]
we arrive at
\begin{align*}
    \norm{u_n-u_k}_{L_{1}(\mathbb{R}^{d})} &\leq \frac{2\ve}{1-\tau G(0)} \norm{f_n}_{TV(\mathbb{R}^{d})} + \frac{1}{1-\tau G(0)}\norm{f_n-f_k}_{L_{1}(\mathbb{R}^{d})}\\
    &\quad+ \frac{1}{(1-\tau G(0))^{2}}
\frac{2 d \tau}{\ve} \sup_{s\in I(f_{n})}\left[\sqrt{\vf'_{n}(s)} - \sqrt{\vf'_{k}(s)}\right]^2\norm{f_n}_{TV(\mathbb{R}^{d})}.
\end{align*}
Choosing $\ve = \sqrt{d\tau}\sup_{s\in I(f_{n})}\left|\sqrt{\vf'_{n}(s)} - \sqrt{\vf'_{k}(s)}\right|$ we get
\[
\norm{u_n-u_k}_{L_{1}(\mathbb{R}^{d})} \leq \frac{1}{1-\tau G(0)}\norm{f_n-f_k}_{L_{1}(\mathbb{R}^{d})} + \frac{4}{(1-\tau G(0))^{2}}
\sqrt{ d \tau}\sup_{s\in I(f_{n})}\left|\sqrt{\vf'_{n}(s)} - \sqrt{\vf'_{k}(s)}\right|\norm{f_n}_{TV(\mathbb{R}^{d})}.
\]
In view of (\ref{fnmbounds}) there holds
\[
\sup_{s\in I(f_{n})}\left|\sqrt{\vf'_{n}(s)} - \sqrt{\vf'_{k}(s)}\right|\norm{f_n}_{TV(\mathbb{R}^{d})} \leq  n \norm{\nabla \eta}_{L_{1}(\mathbb{R}^{d})}  \norm{f}_{L_{1}(\mathbb{R}^{d})} \sup_{s\in [-m,\frac{m}{1-\tau G(0)}]}\left|\sqrt{\vf'_{n}(s)} - \sqrt{\vf'_{k}(s)}\right|.
\]
Uniform convergence of $\vf'_n$ on compact subsets implies that
\[
\forall n>0 \hd  \forall m>0 \hd  \exists k=k(n,m) \hd \sup_{\tilde{k} \geq k} \sup_{s\in [-m,\frac{m}{1-\tau G(0)}]}\left|\sqrt{\vf'_{k}(s)} - \sqrt{\vf'_{\tilde{k}}(s)}\right| \leq \frac{1}{n^{2}}.
\]
Since we may choose $k$ as an increasing function of $m$, we may also invert it to obtain
\[
\forall n>0 \hd  \hd \forall  k \hd  \exists m=m(n,k) \hd \sup_{\tilde{k}\geq k} \sup_{s\in [-m,\frac{m}{1-\tau G(0)}]}\left|\sqrt{\vf'_{k}(s)} - \sqrt{\vf'_{\tilde{k}}(s)}\right| \leq \frac{1}{n^{2}}.
\]
Choosing $k=n$ we arrive at
\[
 \forall n  \hd  \exists m=m(n) \hd \sup_{\tilde{k}\geq n}  \sup_{s\in [-m,\frac{m}{1-\tau G(0)}]}\left|\sqrt{\vf'_{n}(s)} - \sqrt{\vf'_{\tilde{k}}(s)}\right| \leq \frac{1}{n^{2}}.
\]
Hence, for such a choice of $f_n$, $u_n$ is a Cauchy sequence in $L_{1}(\mathbb{R}^d)$. Therefore, there exists a function $\tilde{u}\in L_1(\R^d)$ such that $u_n \rightarrow \tilde{u}$ in $L_1$ and, passing to a subsequence, we may assume that $u_n \rightarrow \tilde{u}$ almost everywhere. Since $G$ and $p$ are continuous, $G(p(u_{n})) \rightarrow G(p(\tilde{u}))$ almost everywhere in $\mathbb{R}^d$. Since $G(p(u_{n}))$ is bounded, it has a weakly-star convergent subsequence in $L_{\infty}(\mathbb{R}^d)$. We deduce that $(u_{n})_{+}G(p(u_{n}))\rightharpoonup (\tilde{u})_{+}G(p(\tilde{u}))$ in $L_{1}(\mathbb{R}^d)$. Consequently, from (\ref{cb}), also $\Delta \vf_{n}(u_{n})$ converges weakly in $L_{1}(\mathbb{R}^d)$.
We will show that it converges to $\Delta\vf(\tilde{u})$, where the Laplacian is understood in the sense of distributions. Recall that $\vf_{n}$ converges to $\vf$ pointwise and uniformly on compact subsets of $\mathbb{R}$. Since $u_{n} \rightarrow \tilde{u}$ almost everywhere and $\vf$ is continuous, we have $\vf(u_{n}) \rightarrow \vf(\tilde{u})$ almost everywhere. Moreover,
\[
\abs{\vf_n(u_n(x)) - \vf(u_n(x))} \leq \sup_{t \in I(f_n)} \abs{\vf_{n}(t) - \vf(t)}.
\]
 Thus, for a.e.\ $x\in\R^d$,
\[
\abs{\vf_n(u_n(x)) - \vf(\tilde{u}(x))} \leq \abs{\vf_n(u_n(x)) - \vf(u_n(x))} + \abs{\vf(u_n(x)) - \vf(\tilde{u}(x))} \rightarrow 0 \hd
\]
and $\vf_{n}(u_n) \rightarrow  \vf(\tilde{u})$ a.e.
Finally, for any $\Phi \in C^{\infty}_{0}(\mathbb{R}^d)$
\[
\tau\int_{\mathbb{R}^d} \vf_n(u_n) \Delta \Phi dx = \int_{\mathbb{R}^d} u_n \Phi dx - \tau\int_{\mathbb{R}^d} (u)_{+}G(p(u))\Phi dx - \int_{\mathbb{R}^d} f \Phi dx,
\]
and the right-hand side converges, so also the left-hand side is convergent and from the pointwise limit we have
\[
\int_{\mathbb{R}^d} \vf_n(u_n) \Delta \Phi dx \rightarrow \int_{\mathbb{R}^d} \vf(\tilde{u}) \Delta \Phi dx.
\]
Thus, from the uniqueness of the limit $\Delta \vf_{n}(u_{n})$ converges weakly in $L_{1}(\mathbb{R}^d)$ to $\Delta \vf(\tilde{u})$. Hence, $\tilde{u}$ satisfies (\ref{Pw}) in the sense of distributions and from the uniqueness of the solution $\tilde{u} = u$.
\end{proof}
As a consequence of the previous approximation result, we can deduce that the properties demonstrated in the previous section hold also for the un-regularised $\vf$. More precisely, we have:
\begin{prop}\label{Atheo}
Let $\vf,G,p$ satisfy (\ref{A1})--(\ref{A3}) and let $\tau \in (0,\frac{1}{G(0)})$.
Define
\eqq{
A(u)= -\Delta \vf(u) - (u)_{+}G(p(u)), \hd D(A) = \{u\in L_{1}(\mathbb{R}^d): -\Delta \vf(u) \in L_{1}(\mathbb{R}^d)\}.
}{Al1def}
Then $R(I+\tau A) = L_{1}(\mathbb{R}^{d})$ and the results of Proposition \ref{image}, \ref{exi}, and  \ref{acr} hold for $A$ defined by \eqref{Al1def}.
\end{prop}
\begin{proof}
The fact that $R(I+\tau A) = L_{1}(\mathbb{R}^{d})$ follows from Proposition \ref{exi}. The results 2. and 3. of Proposition~\ref{image} follow from Proposition \ref{apro} by weak lower semicontinuity of norm. To show that the results of Proposition~\ref{acr} hold for $A$ defined by \eqref{Al1def}, we take for $i=1,2$ $f_i \in L_{1}(\mathbb{R}^d)$ and $\vf_i,G_i,p_i$  satisfying (\ref{A1})--(\ref{A3}) and denote by $\vf_{i,n}$ the approximation of $\vf_i$ given by (\ref{vfapro}).  Further, we take sequences $f_{i,n} \in L_{1}(\mathbb{R}^d)\cap L_{\infty}(\mathbb{R}^d)\cap BV(\mathbb{R}^d)$ with $f_{i,n} \rightarrow f_{i}$ in $L_{1}(\mathbb{R}^d)$ from Proposition \ref{apro}. Then we denote $u_{i,n} := (I+\tau A_{\vf_{i,n},G_i,p_i})^{-1}f_{i,n}$. From (\ref{acretive}) we obtain
\[
\norm{u_{1,n} - u_{2,n}}_{L_{1}(\mathbb{R}^d)} \leq \frac{1}{1-\tau G(0)}\norm{f_{1,n} - f_{2,n}}_{L_{1}(\mathbb{R}^d)},
\]
and passing to the limit $n\to\infty$ we arrive at
\[
\norm{u_{1} - u_{2}}_{L_{1}(\mathbb{R}^d)} \leq \frac{1}{1-\tau G(0)}\norm{f_{1} - f_{2}}_{L_{1}(\mathbb{R}^d)}.
\]
Let us now suppose that $f_1 \in  L_{1}(\mathbb{R}^{d})\cap L_{\infty}(\mathbb{R}^d)\cap BV(\R^d)$. Then, estimate (\ref{tv}) follows easily from accretivity and translation invariance of $A$. Finally, taking $f_{1,n} \equiv f_1$ (which also gives $u_{1,n} \rightarrow u_1$ in $L_{1}(\mathbb{R}^d)$), estimate (\ref{estepsilon}) gives
\[
\int_{\mathbb{R}^{d}}\int_{\mathbb{R}^{d}}\abs{u_{1,n}-u_{2,n}}q^{\ve}(x-y)dydx
\]
\[
\leq \frac{1}{(1-\tau G(0))^{2}}
\frac{2 d \tau}{\ve} \sup_{s\in I(f_{1})}\left[\sqrt{\vf'_{1,n}(s)} - \sqrt{\vf'_{2,n}(s)}\right]^2\norm{f_{1}}_{TV(\mathbb{R}^{d})} 
\]
\[
+ \frac{1}{1-\tau G(0)}\int_{\mathbb{R}^{d}}\int_{\mathbb{R}^{d}}\abs{f_{1}-f_{2,n}}q^{\ve}(x-y)dydx
\]
\[
+
 \frac{\tau}{(1-\tau G(0))^{2}} \left(\norm{G'_2}_{L_{\infty}(I_{p_1,p_2}(f_1))}\sup_{s \in I(f_{1})}\abs{p_{2}(s)-p_{1}(s)}+\norm{G_2-G_1}_{L_{\infty}(I_{p_1}(f_{1}))}\right)\norm{f_{1}}_{L_{1}(\mathbb{R}^{d})}.
\]
Passing to the limit with $n$, we obtain the claim.
\end{proof}

\section{Stability results}
\label{sec:Stability}
We are now ready to establish the main results of this paper. Let us first recall the classical theorem of Crandall-Liggett~\cite{CL} regarding the existence of mild solutions to the abstract evolution equation $u_t= Au$.
\begin{theo}\cite{CL}\label{CL}
    Let $\omega \geq 0$ and $A$ be an $\omega$-accretive operator in a Banch space $X$ with $R(I+\tau A) = X$ for sufficiently small positive $\tau$. Then, for any $u_0 \in X$ the limit
\[
S_{t}(A)u_0:= \lim_{n\rightarrow \infty}(I+\frac{t}{n} A)^{-n}u_0
\]
exists uniformly on compact subsets of $[0,\infty)$. Moreover, the family of operators $S_{t}(A)$, $t > 0$, is a strongly continuous semigroup of contractive mappings of $D(A) \subset X$ and
$u(t):= S_{t}(A)u_0$ is the unique mild solution of $u_t = Au$ with $u(0)=u_0$.
Furthermore, if $u_0 \in D(A)$, then
\[
\norm{u(t)-(I+\frac{t}{n} A)^{-n}u_0}_{X} \leq \frac{t}{\sqrt{n}}\norm{A u_{0}}.
\]
\end{theo}
Note that Proposition \ref{Atheo} establishes that operator $A$ defined in (\ref{Al1def}) is $G(0)$-accretive and satisfies the assumptions of Theorem \ref{CL} in $L_{1}(\mathbb{R}^d)$.
\begin{remark}\label{re1}
The mild solution $u$ to (\ref{r1}) may be constructed in the following approximation procedure. Let
\[
u_{n} = (I+\tau A)^{-1}u_{n-1} \m{ and }u_{0} = u^{0},\; n\in\mathbb{N}
\]
where $A$ is defined in (\ref{Al1def}), $\tau \in (0,\frac{1}{G(0)})$ and define the piecewise constant function
\[
u^n(t) = u_{n} \m{ for }t \in [n\tau,(n+1)\tau),\; n\in\mathbb{N}.
\]
Then
\[
u^{n}(n\tau) = (I+\tau A)^{-1}u^{n}((n-1)\tau) \hd \m{ and } \hd u(t) = \lim_{n\rightarrow \infty}u^n(t).
\]
\end{remark}
\begin{theo}\label{thm:Main1}
Consider the two problems
\eqq{
\partial_{t} u = \Delta (\varphi_i (u)) +(u)_{+
}G_i(p_i(u))\hd  \m{ in } \mathbb{R}^d \times (0,T), \hd  u(0) = u^{0}_i  \m{ in } \mathbb{R}^{d}, \hd i = 1,2
}{problems}
such that for $i=1,2$  $u^0_i \in L_{1}(\mathbb{R}^d)$, and  $\vf_i, G_i,p_i$ satisfy  (\ref{A1})--(\ref{A3}).
Let $u_1, u_2$ denote the mild solutions to (\ref{problems}) for $i=1,2$ respectively, given by Theorem \ref{CL}. 
Suppose that $u_1^{0} \in L_{\infty}(\mathbb{R}^d)\cap BV(\mathbb{R}^{d})$. Then, for almost all $t \in (0,T)$ there holds
\[
\norm{u_1(t)-u_2(t)}_{L_{1}(\mathbb{R}^{d})}
 \leq  e^{tG(0)}\norm{u^0_1-u^0_2}_{L_{1}(\mathbb{R}^{d})}
+4 \sqrt{d t } e^{tG(0)} \norm{u^{0}_1}_{TV(\mathbb{R}^{d})}
\sup_{s\in I^{t}(u^0_1)}\left|\sqrt{\vf'_{1}(s)} - \sqrt{\vf'_{2}(s)}\right|
\]
\[
+
 te^{tG(0)} \norm{G'_2}_{L_{\infty}(I_{p_1,p_2}^t(u^0_1))}
 \sup_{s\in I^{t}(u^0_1) }\abs{p_{2}(s)-p_{1}(s)} \norm{u^{0}_1}_{L_{1}(\mathbb{R}^{d})}
 +  te^{tG(0)}\norm{G_2-G_1}_{L_{\infty}(I_{p_1}^t(u^0_1))}\norm{u^{0}_1}_{L_{1}(\mathbb{R}^{d})},
\]
where $G(0)=\max\{G_1(0),G_{2}(0)\}$ and
\[
I^{t}(u^0_1):= \left[-\norm{(u^0_1)_{-}}_{L_{\infty}(\mathbb{R}^d)},e^{tG(0)}\norm{(u^0_1)_+}_{L_{\infty}(\mathbb{R}^d)}\right],
\]
\[
I_{p_1}^t(u^0_1):= \left[p_1\left(-\norm{(u^0_1)_{-}}_{L_{\infty}(\mathbb{R}^d)}\right),p_1\left(e^{tG(0)}\norm{(u^0_1)_+}_{L_{\infty}(\mathbb{R}^d)}\right)\right],
\]
\[
I_{p_1,p_2}^t(u^0_1):= \left[\min_{i=1,2}p_i\left(-\norm{(u^0_1)_{-}}_{L_{\infty}(\mathbb{R}^d)}\right),\max_{i=1,2}p_i\left(e^{tG(0)}\norm{(u^0_1)_+}_{L_{\infty}(\mathbb{R}^d)}\right)\right].
\]

Furthermore, similar stability result holds for the approximations $u_{i,n}$, defined in Remark \ref{re1} by $
u_{i,n} = (I+\tau A_{\vf_i,G_i,p_i})^{-1}u_{i,n-1} \m{ and }u_{i,0} = u_i^{0},\; n\in\mathbb{N}
$, namely
\[
\norm{u_{1,n}-u_{2,n}}_{L_{1}(\mathbb{R}^{d})}
 \leq  e^{n\tau G(0)}\norm{u^0_1-u^0_2}_{L_{1}(\mathbb{R}^{d})}
+4 \sqrt{d n \tau } e^{n\tau  G(0)} \norm{u^{0}_1}_{TV(\mathbb{R}^{d})}
\sup_{s\in I^{n\tau}(u^0_1)}\left|\sqrt{\vf'_{1}(s)} - \sqrt{\vf'_{2}(s)}\right|
\]
\eqq{
+
 n\tau e^{n \tau G(0)} \norm{G'_2}_{L_{\infty}(I_{p_1,p_2}^{n\tau}(u^0_1))}
 \sup_{s\in I^{n\tau}(u^0_1) }\abs{p_{2}(s)-p_{1}(s)} \norm{u^{0}_1}_{L_{1}(\mathbb{R}^{d})}
 +  n\tau e^{n\tau G(0)}\norm{G_2-G_1}_{L_{\infty}(I_{p_1}^{n\tau}(u^0_1))}\norm{u^{0}_1}_{L_{1}(\mathbb{R}^{d})}.
}{discrstab}
Finally, for $u^0_i \in D(A)$ as in \eqref{Al1def} we have
\eqq{
\norm{u_i(t)-u_{i,n}}_{L_{1}(\mathbb{R}^d)} \leq \frac{t}{\sqrt{n}}\norm{A u_{0}^i}.
}{CLform}
\end{theo}
\begin{proof}
Estimate (\ref{CLform}) follows from the Crandall-Liggett formula (Theorem \ref{CL}). To show the stability result, let us begin with the discretisation.  We set
\[
u^n_i(t) = u_{i, n} \m{ for }t \in [n\tau,(n+1)\tau),\ n~\in~\mathbb{N}
\]
where $u_{i,n} = (I+\tau A_{\vf_i,G_i,p_i})^{-1}u_{i,n-1} \m{ and }u_{i,0} = u^{0}_i, \m{ for }i=1,2.$
Then $u^{n}_i(n\tau) = (I+\tau A)^{-1}u^{n}_{i}((n-1)\tau)$. For $t \in (0,T]$ we choose $\tau \in (0,\frac{1}{G(0)})$ and $n \in \mathbb{N}$ such that $t = n\tau$. We apply (\ref{estepsilon}) to obtain

\[
\int_{\mathbb{R}^{d}}\int_{\mathbb{R}^{d}}\abs{u^n_1(n\tau,x)-u^n_2(n\tau,y)}q^{\ve}(x-y)dydx
\]
\[
 \leq \frac{1}{1-\tau G(0)}\int_{\mathbb{R}^{d}}\int_{\mathbb{R}^{d}}\abs{u^n_1((n-1)\tau,x)-u^n_2((n-1)\tau,y)}q^{\ve}(x-y)dydx
\]
\[
+ \frac{1}{(1-\tau G(0))^{2}}
\frac{2 d \tau}{\ve} \sup_{s\in I(u_{1}^{n}((n-1)\tau))}\left[\sqrt{\vf'_{1}(s)} - \sqrt{\vf'_{2}(s)}\right]^2\norm{u^{n}_1((n-1)\tau)}_{TV(\mathbb{R}^{d})}
\]
\eqq{
+
 \frac{\tau\norm{u^{n}_1((n-1)\tau)}_{L_{1}(\mathbb{R}^{d})}}{(1-\tau G(0))^{2}} \left(\norm{G'_2}_{L_{\infty}(I_{p_1,p_2}(u_1^n((n-1)\tau)))}\sup_{s\in I(u_{1}^{n}((n-1)\tau))}\abs{p_{2}(s)-p_{1}(s)}+\norm{G_2-G_1}_{L_{\infty}(I_{p_1}(u_1^n((n-1)\tau)))}\right).
}{ba}
Proposition~\ref{Atheo} implies
\[
\norm{u^{n}_1((n-1)\tau)}_{L_{1}(\mathbb{R}^{d})} \leq (1-\tau G(0))^{-(n-1)}\norm{u^{0}_1}_{L_{1}(\mathbb{R}^{d})}, \hd \norm{(u^{n}_1((n-1)\tau))_+}_{L_{\infty}(\mathbb{R}^{d})} 
\]
\[
\leq (1-\tau G(0))^{-(n-1)}\norm{(u^{0}_1)_+}_{L_{\infty}(\mathbb{R}^{d})},
\]
as well as
\[
\norm{u^{n}_1((n-1)\tau)}_{TV(\mathbb{R}^{d})} \leq (1-\tau G(0))^{-(n-1)}\norm{u^{0}_1}_{TV(\mathbb{R}^{d})}.\]
Using these bounds in (\ref{ba}), iterating the resulting inequality and denoting
 \[
    I^n(u^0_1) := \left[-\norm{(u^0_1)_{-}}_{L_{\infty}(\mathbb{R}^{d})},\frac{1}{(1-\tau G(0))^{n}}\norm{(u^0_1)_+}_{L_{\infty}(\mathbb{R}^{d})}\right],
\]
\[
 I^n_{p_1,p_2}(u^0_1) := \left[\min_{i\in \{1,2\}} p_i\left(-\norm{(u^0_1)_{-}}_{L_{\infty}(\mathbb{R}^{d})}
\right),\max_{i\in \{1,2\}}p_i\left(\frac{\norm{(u^0_1)_+}_{L_{\infty}(\mathbb{R}^d)}}{(1-\tau G(0))^{n}}\right) \right],
\]
    \[
    I^n_{p_i}(u^0_1) := \left[  p_i\left(-\norm{(u^0_1)_{-}}_{L_{\infty}(\mathbb{R}^{d})}
\right),  p_i\left(\frac{\norm{(u^0_1)_+}_{L_{\infty}(\mathbb{R}^d)}}{(1-\tau G(0))^{n}}\right)\right], \hd i = 1,2,
    \]
leads to
\[
\int_{\mathbb{R}^{d}}\int_{\mathbb{R}^{d}}\abs{u^n_1(n\tau,x)-u^n_2(n\tau,y)}q^{\ve}(x-y)dydx
\]
\[
 \leq (1-\tau G(0))^{-n}\int_{\mathbb{R}^{d}}\int_{\mathbb{R}^{d}}\abs{u^0_1(x)-u^0_2(y)}q^{\ve}(x-y)dydx
\]
\[
+ (1-\tau G(0))^{-(n+1)}
\frac{2 d n \tau}{\ve} \sup_{s\in I^n(u^0_1)}\left[\sqrt{\vf'_{1}(s)} - \sqrt{\vf'_{2}(s)}\right]^2\norm{u^{0}_1}_{TV(\mathbb{R}^{d})}
\]
\[
+
 n\tau{(1-\tau G(0))^{-(n+1)}} \left(\norm{G'_2}_{L_{\infty}(I^{n}_{p_1,p_2}(u^0_n))}\sup_{s\in I^n(u^0_1)}\abs{p_{2}(s)-p_{1}(s)} + \norm{G_2-G_1}_{L_{\infty}(I^{n}_{p_1}(u^0_1))}\right)\norm{u^{0}_1}_{L_{1}(\mathbb{R}^{d})}.
 \]
Applying the estimate
\[
\abs{\int_{\mathbb{R}^{d}}\int_{\mathbb{R}^{d}}\abs{f(x)-g(y)}q^{\ve}(x-y)dydx - \norm{f-g}_{L_{1}(\mathbb{R}^d)} } \leq \ve \norm{f}_{TV(\mathbb{R}^d)},
\]
we get
\[
\norm{u^n_1(n\tau)-u^n_2(n\tau)}_{L_{1}(\mathbb{R}^{d})}
 \leq 2\ve (1-\tau G(0))^{-n}\norm{u_1^0}_{TV(\mathbb{R}^d)} + (1-\tau G(0))^{-n}\norm{u^0_1-u^0_2}_{L_{1}(\mathbb{R}^{d})}
\]
\[
+ (1-\tau G(0))^{-(n+1)}
\frac{2 d n \tau}{\ve} \norm{u^{0}_1}_{TV(\mathbb{R}^{d})}\sup_{s\in I^{n}(u^{0}_1)}\left[\sqrt{\vf'_{1}(s)} - \sqrt{\vf'_{2}(s)}\right]^2
\]
\[
+
n \tau{(1-\tau G(0))^{-(n+1)}} \left(\norm{G'_2}_{L_{\infty}(I^{n}_{p_1,p_2}(u^0_n))}\sup_{s\in I^n(u^0_1)}\abs{p_{2}(s)-p_{1}(s)} + \norm{G_2-G_1}_{L_{\infty}(I^{n}_{p_1}(u^0_1))}\right)\norm{u^{0}_1}_{L_{1}(\mathbb{R}^{d})}.
\]
Choosing $\ve =  \sup_{s\in I^{n}(u^0_1)}\left|\sqrt{\vf'_{1}(s)} - \sqrt{\vf'_{2}(s)}\right| \sqrt{dn\tau}$ and recalling $t=n\tau$ we obtain
\[
\norm{u^n_1(t)-u^n_2(t)}_{L_{1}(\mathbb{R}^{d})}
 \leq  (1- \frac{t}{n}G(0))^{-n}\norm{u^0_1-u^0_2}_{L_{1}(\mathbb{R}^{d})}
\]
\[
+4 \sqrt{d t } (1-\frac{t}{n} G(0))^{-(n+1)}
\norm{u^{0}_1}_{TV(\mathbb{R}^{d})}\sup_{s\in I^{n}(u^0_1)}\left|\sqrt{\vf'_{1}(s)} - \sqrt{\vf'_{2}(s)}\right|
\]
\eqq{
+
 t{(1-\frac{t}{n} G(0))^{-(n+1)}} \left(\norm{G'_2}_{L_{\infty}(I^{n}_{p_1,p_2}(u^0_n))}\sup_{s\in I^n(u^0_1)}\abs{p_{2}(s)-p_{1}(s)} + \norm{G_2-G_1}_{L_{\infty}(I^{n}_{p_1}(u^0_1))}\right)\norm{u^{0}_1}_{L_{1}(\mathbb{R}^{d})}.
}{discrstab1}
Passing with $n$ to infinity and recalling that, due to Theorem \ref{CL}, $u_1^n, u_2^n$ converge respectively to $u_1, u_2$ the mild solutions of (\ref{problems}) uniformly on compact subsets of $\mathbb{R}_{+}$, we arrive at
\[
\norm{u_1(t)-u_2(t)}_{L_{1}(\mathbb{R}^{d})}
 \leq  e^{tG(0)}\norm{u^0_1-u^0_2}_{L_{1}(\mathbb{R}^{d})}
+4 \sqrt{d t } e^{tG(0)} \norm{u^{0}_1}_{TV(\mathbb{R}^{d})}
\sup_{s\in I^{t}(u^0_1)}\left|\sqrt{\vf'_{1}(s)} - \sqrt{\vf'_{2}(s)}\right|
\]
\[
+
 te^{tG(0)} \norm{G'_2}_{L_{\infty}(I_{p_1,p_2}^t(u^0_1))}
 \sup_{s\in I^{t}(u^0_1) }\abs{p_{2}(s)-p_{1}(s)} \norm{u^{0}_1}_{L_{1}(\mathbb{R}^{d})}
 +  te^{tG(0)}\norm{G_2-G_1}_{L_{\infty}(I_{p_1}^t(u^0_1))}\norm{u^{0}_1}_{L_{1}(\mathbb{R}^{d})}.
\]
The estimate (\ref{discrstab}) follows from the fact that the right-hand side of (\ref{discrstab1}) converges increasingly.
\end{proof}

Finally, from Theorem \ref{thm:Main1} we deduce the following corollary, which is the main result of this paper (cf. Theorem~\ref{thm:MainIntro}).

\begin{coro} \label{thm:Main2}
Let $u_1,u_2$ be two mild solutions to (\ref{r1}), with $G_i$ satisfying (\ref{A2})  and $G_i \in W^{1,\infty}(\mathbb{R})$, $u^0_i \geq 0$, $u^0_i \in L_{1}(\mathbb{R}^{d})$, i =1,2. Assume additionally that $u^0_1 \in L_{\infty}(\mathbb{R}^d) \cap BV(\mathbb{R}^{d})$,  and
\[
\vf_{i}(t) = sign(t)|t|^{\gamma_i}, \hd p_{i}(t) = \frac{\gamma_i}{\gamma_i - 1}|t|^{\gamma_i - 1}, \hd \m{ for } i=1,2, \hd \gamma_2 > \gamma_1 >1.
\]
Then
\[
\norm{u_1(t)-u_2(t)}_{L_{1}(\mathbb{R}^{d})}
 \leq  e^{tG(0)}\norm{u^0_1-u^0_2}_{L_{1}(\mathbb{R}^{d})}  +  te^{tG(0)} \norm{G_2-G_1}_{L_{\infty}(\mathbb{R})} \norm{u^{0}_1}_{L_{1}(\mathbb{R}^{d})}
\]
\[
+4 \sqrt{d t } e^{tG(0)} \norm{u^{0}_1}_{TV(\mathbb{R}^{d})}
|\gamma_2-\gamma_1|\left(\frac{\sqrt{\gamma_1}}{\gamma_2-1}+\frac{M^{\frac{\gamma_1-1}{2}}}{\sqrt{\gamma_2}+\sqrt{\gamma_1}} +\frac{\sqrt{\gamma_2}M^{\frac{\gamma_2-1}{2}}|\ln M|}{2}\right)
\]
\[
+
 te^{tG(0)} \norm{G'_2}_{L_{\infty}(\mathbb{R})} |\gamma_2-\gamma_1|\left
(\frac{\gamma_1}{(\gamma_2-1)(\gamma_1-1)} +\frac{M^{\gamma_1-1}}{(\gamma_2-1)(\gamma_1-1)} + \frac{\gamma_2}{\gamma_2-1} M^{\gamma_2-1}|\ln M| \right)\norm{u^{0}_1}_{L_{1}(\mathbb{R}^{d})},
 \]
where $M = e^{tG(0)}\norm{u^{0}_1}_{L_{\infty}(\mathbb{R}^d)}$.
\end{coro}
\begin{proof}[Proof of Corollary \ref{thm:Main2}]
Let us find the maximum on $[0,e^{tG(0)}||u^0_1||_{L_{\infty}(\mathbb{R}^{N})}]$ of the function $g$ defined below
\[
g(s) := \sqrt{\varphi'_1(s)} - \sqrt{\vf_2'(s)} = \sqrt{\gamma_1}s^{\frac{\gamma_1-1}{2}} - \sqrt{\gamma_2}s^{\frac{\gamma_2-1}{2}}.
\]
Clearly $g(0) = 0$ and for $M:=e^{tG(0)}||u^0_1||_{L_{\infty}(\mathbb{R}^{N})}$
\[
\abs{g(M)} \leq \abs{\sqrt{\gamma_1}-\sqrt{\gamma_2}}M^{\frac{\gamma_1-1}{2}} + \sqrt{\gamma_2}\abs{M^{\frac{\gamma_1-1}{2}} - M^{\frac{\gamma_2-1}{2}}}
\leq |\gamma_2-\gamma_1|\left(\frac{M^{\frac{\gamma_1-1}{2}}}{\sqrt{\gamma_2}+\sqrt{\gamma_1}} +\frac{\sqrt{\gamma_2}M^{\frac{\gamma_2-1}{2}}|\ln M|}{2}\right).
\]
Since
\[
g'(s) = \sqrt{\gamma_1}\frac{\gamma_1-1}{2}s^{\frac{\gamma_1-3}{2}} - \sqrt{\gamma_2}\frac{\gamma_2-1}{2}s^{\frac{\gamma_2-3}{2}} = \frac{1}{2}s^{\frac{\gamma_1-3}{2}}\left(\sqrt{\gamma_1}(\gamma_1-1) -\sqrt{\gamma_2}(\gamma_2-1)s^{\frac{\gamma_2-\gamma_1}{2}} \right),
\]
$g$ attains a local maximum in
\[
s_0 = \left(\Upsilon(\gamma_1,\gamma_2)\right)^{\frac{2}{\gamma_2-\gamma_1}}, \hd \m{ where } \hd \Upsilon(\gamma_1,\gamma_2) := \frac{\sqrt{\gamma_1}(\gamma_1-1)}{\sqrt{\gamma_2}(\gamma_2-1)},
\]
with value
\begin{align*}
    g(s_0) = \sqrt{\gamma_1}\Upsilon(\gamma_1,\gamma_2)^{\frac{\gamma_1-1}{\gamma_2 - \gamma_1}} - \sqrt{\gamma_2}\Upsilon(\gamma_1,\gamma_2)^{\frac{\gamma_2-1}{\gamma_2 - \gamma_1}}
    &= \Upsilon(\gamma_1,\gamma_2)^{\frac{\gamma_1-1}{\gamma_2 - \gamma_1}}\left(\sqrt{\gamma_1} - \sqrt{\gamma_2}\Upsilon(\gamma_1,\gamma_2)\right)\\
    &= \Upsilon(\gamma_1,\gamma_2)^{\frac{\gamma_1-1}{\gamma_2 - \gamma_1}} \frac{\sqrt{\gamma_1}}{\gamma_2-1}(\gamma_2-\gamma_1).
\end{align*}
We claim that for any $1 \leq \gamma_1 < \gamma_2$, there holds
\eqq{
\Upsilon(\gamma_1,\gamma_2)^{\frac{\gamma_1-1}{\gamma_2 - \gamma_1}} \leq 1.
}{Aest}
Indeed, we have
\[
\Upsilon(\gamma_1,\gamma_2)^{\frac{\gamma_1-1}{\gamma_2 - \gamma_1}} = \left(1+\frac{\sqrt{\gamma_1}(\gamma_1-1) - \sqrt{\gamma_2}(\gamma_2-1)}{\sqrt{\gamma_2}(\gamma_2-1)}\right)^{\frac{\sqrt{\gamma_2}(\gamma_2-1)}{{\sqrt{\gamma_1}(\gamma_1-1) - \sqrt{\gamma_2}(\gamma_2-1)}}\frac{{\sqrt{\gamma_1}(\gamma_1-1) - \sqrt{\gamma_2}(\gamma_2-1)}}{\sqrt{\gamma_2}(\gamma_2-1)}\frac{\gamma_1-1}{\gamma_2-\gamma_1}},
\]
and
\begin{align*}
  \frac{{\sqrt{\gamma_1}(\gamma_1-1) - \sqrt{\gamma_2}(\gamma_2-1)}}{\sqrt{\gamma_2}(\gamma_2-1)}\frac{\gamma_1-1}{\gamma_2-\gamma_1} &= \frac{(\gamma_1-1)(\sqrt{\gamma_1} - \sqrt{\gamma_2}) + \sqrt{\gamma_2}(\gamma_1-\gamma_2)}{\sqrt{\gamma_2}(\gamma_2-1)(\gamma_2-\gamma_1)}(\gamma_1-1)\\[1em]
  &=-\frac{(\gamma_1-1)^2}{\sqrt{\gamma_2}(\gamma_2-1)(\sqrt{\gamma_1}+\sqrt{\gamma_2})} - \frac{\gamma_1-1}{\gamma_2-1},
\end{align*}
which is negative.
Hence, since
\[
\left(1+\frac{\sqrt{\gamma_1}(\gamma_1-1) - \sqrt{\gamma_2}(\gamma_2-1)}{\sqrt{\gamma_2}(\gamma_2-1)}\right)^{\frac{\sqrt{\gamma_2}(\gamma_2-1)}{{\sqrt{\gamma_1}(\gamma_1-1) - \sqrt{\gamma_2}(\gamma_2-1)}}} \longrightarrow e \m{ as } \gamma_1\longrightarrow \gamma_2,
\]
we infer (\ref{Aest}).
Therefore, we have
\[
|g(s_0)| \leq |\gamma_2-\gamma_1|\frac{\sqrt{\gamma_1}}{\gamma_2-1},
\]
and consequently
\begin{equation}\label{eq:GammaEstimate1}
    \sup_{s\in [0,M]}|g(s)| \leq |\gamma_2-\gamma_1|\left(\frac{\sqrt{\gamma_1}}{\gamma_2-1}+\frac{M^{\frac{\gamma_1-1}{2}}}{\sqrt{\gamma_2}+\sqrt{\gamma_1}} +\frac{\sqrt{\gamma_2}M^{\frac{\gamma_2-1}{2}}|\ln M|}{2}\right).
\end{equation}

We deal similarly with the function
\[
h(s) := p_1(s) - p_2(s) = \frac{\gamma_1}{\gamma_1-1}s^{\gamma_1-1} - \frac{\gamma_2}{\gamma_2-1}s^{\gamma_2-1}.
\]
We have $h(0)=0$ and
\[
|h(M)| \leq |\gamma_2-\gamma_1|\frac{M^{\gamma_1-1}}{\gamma_2-1}\left(\frac{1}{\gamma_1-1} + \gamma_2 M^{\gamma_2-\gamma_1}|\ln M|\right).
\]
The critical point $s_0$ now equals
\[
s_0 = \left(\frac{\gamma_1}{\gamma_2}\right)^{\frac{1}{\gamma_2-\gamma_1}},
\]
for which
\[
\abs{h(s_0)} = \left(\frac{\gamma_1}{\gamma_2}\right)^{\frac{\gamma_1-1}{\gamma_2-\gamma_1}} \left(\frac{\gamma_1}{\gamma_1-1}-\frac{\gamma_1}{\gamma_2-1}\right) = |\gamma_2-\gamma_1|\frac{\gamma_1}{(\gamma_2-1)(\gamma_1-1)}\left(\frac{\gamma_1}{\gamma_2}\right)^{\frac{\gamma_1-1}{\gamma_2-\gamma_1}} 
\]
\[
\leq  |\gamma_2-\gamma_1|\frac{\gamma_1}{(\gamma_2-1)(\gamma_1-1)}.
\]
It follows that
\begin{equation}\label{eq:GammaEstimate2}
 \sup_{s\in [0,M]}\abs{p_{2}(s)-p_{1}(s)} \leq |\gamma_2-\gamma_1|\left
(\frac{\gamma_1}{(\gamma_2-1)(\gamma_1-1)} +\frac{M^{\gamma_1-1}}{(\gamma_2-1)(\gamma_1-1)} + \frac{\gamma_2}{\gamma_2-1} M^{\gamma_2-1}|\ln M| \right).
\end{equation}
The results follows by using~\eqref{eq:GammaEstimate1} and~\eqref{eq:GammaEstimate2} in Theorem~\ref{thm:Main1}.
\end{proof}

\vspace{0.5em}
\noindent\textbf{Acknowledgements.}
T.D. acknowledges the support of the National Science Centre, Poland, project no. 2023/51/D/ST1/02316. P.G. was supported by the National Science Centre Poland grant UMO-2024/54/A/ST1/00159.
B.M. and Z.S. acknowledge support from the Excellence Initiative - Research University programme (2020-2026) at the University of Warsaw.
K.R and A.W-K. were partly supported by the grant Sonata Bis UMO-2020/38/E/ST1/00469, National Science Centre, Poland.

\end{document}